\documentclass[colorinlistoftodos,12pt,bibliography=totoc]{article}

\usepackage[utf8]{inputenc}

\usepackage[utf8]{inputenc}
\usepackage{amsmath,amsfonts,amsthm,epsfig,latexsym,graphicx,amssymb,xfrac,faktor}
\usepackage{enumerate} 
\usepackage[english]{babel}

\usepackage{bbm} 

\usepackage{mathtools}
\usepackage{stmaryrd} 

\usepackage[dvipsnames]{xcolor} 
\usepackage{tikz-cd}
\usetikzlibrary{arrows}

\usepackage{graphpap} 
\usepackage[
  colorlinks = true,
  urlcolor = RoyalBlue, 
  linkcolor = RoyalBlue, 
  citecolor = RoyalBlue 
]{hyperref}

\usepackage{sectsty} 
\sectionfont{\centering\color{RoyalBlue}} 
\subsectionfont{\centering\color{RoyalBlue}} 

\DeclareTextFontCommand{\emph}{\color{RoyalBlue}\em} 

\usepackage{ocgx2}
\usepackage{fancyhdr} 
\usepackage{theoremref}
\usepackage{xargs}     

\DeclareFontFamily{U}{dutchcal}{\skewchar \font =45}
\DeclareFontShape{U}{dutchcal}{m}{n}{<-> dutchcal-r}{}
\DeclareFontShape{U}{dutchcal}{b}{n}{<-> dutchcal-b}{}
\DeclareMathAlphabet{\dcal}{U}{dutchcal}{m}{n}
\SetMathAlphabet{\dcal}{bold}{U}{dutchcal}{b}{n}

\newcommand{\RR}{\mathbb{R}}
\newcommand{\NN}{\mathbb{N}}

\newcommand{\ZZ}{\mathbb{Z}}
\newcommand{\QQ}{\mathbb{Q}}
\newcommand{\TT}{\mathbb{T}}

\newcommand{\A}{\mathcal{A}}
\newcommand{\Class}{\mathcal{C}}
\newcommand{\germ}{\dcal{g}}
\newcommand{\sell}{\ell^0}

\newcommand{\simeqd}{\mathrel{\rotatebox[origin=c]{-90}{$\simeq$}}}

\newcommand{\bfrac}[2]{{
  \raisebox{0.1em}{\scalebox{0.85}{$#1$}}/
  \raisebox{-0.1em}{\scalebox{0.85}{$#2$}}
}}

\newcommand{\lsfrac}[2]{{\scalebox{1.3}{$\frac{#1}{#2}$}}}

\DeclareMathOperator{\symbolehmod}{mod}
\newcommand{\symbolelmod}{\mathop{\symbolehmod}\limits^{\mkern-10mu\smash{\raisebox{-3.5pt}{\rule{1.5em}{0.4pt}}}}}
\newcommand{\symbolesmalllmod}{\mathop{\symbolehmod}\limits^{\mkern-6mu\smash{\raisebox{-2.5pt}{\rule{1.1em}{0.4pt}}}}}
\newcommand{\symbolermod}{\underline{\symbolehmod}}
\newcommand{\hmod}[1]{\,\symbolehmod(\hspace{-1px}{#1}\hspace{-1px})}
\newcommand{\lmod}[1]{\,\symbolelmod(\hspace{-1px}{#1}\hspace{-1px})}
\newcommand{\rmod}[1]{\,\symbolermod(\hspace{-1px}{#1}\hspace{-1px})}

\newcommand{\prmod}[1]{\symbolermod(\hspace{-1px}{#1}\hspace{-1px})}
\newcommand{\srmod}{\symbolermod}

\newcommand{\vlmod}[1]{\symbolesmalllmod(\hspace{-1px}{#1}\hspace{-1px})}
\newcommand{\smod}[1]{\,(\mathrm{mod}#1)}

\newcommand{\PS}{\mathcal{PS}}
\DeclareMathOperator{\pL}{pL}

\newcommand{\vertiii}[1]{{\left\vert\kern-0.25ex\left\vert\kern-0.25ex\left\vert #1 
    \right\vert\kern-0.25ex\right\vert\kern-0.25ex\right\vert}}

\newcommand{\wb}{\overline}
\newcommand{\wt}{\widetilde}
\newcommand{\wh}{\widehat}

\newcommand{\lint}{\llbracket}
\newcommand{\rint}{\rrbracket}
\newcommand{\intint}[1]{{\lint #1 \rint}}        

\makeatletter
\def\@tvsp{\mathchoice{{}\mkern-4.5mu}{{}\mkern-4.5mu}{{}\mkern-2.5mu}{}}
\def\ltrivert{\left|\@tvsp\left|\@tvsp\left|}
\def\rtrivert{\right|\@tvsp\right|\@tvsp\right|}

\def\ldrivert{\left|\@tvsp\left|}
\def\rdrivert{\right|\@tvsp\right|}
\makeatother

\newcommand{\recto}{\rightsquigarrow}              
\newcommand{\rectot}{\leftrightsquigarrow}

\newcommand{\Rec}{\mathcal{R}}      

\newcommand{\arrow}{\circlearrowleft}

\newcommand{\col}{\hspace{0.1em}{:}\hspace{0.1em}}

\DeclareMathOperator{\id}{id}
\DeclareMathOperator{\tor}{tors}

\DeclareMathOperator{\rad}{rad}

\DeclareMathOperator{\im}{im}
\DeclareMathOperator{\len}{len}

\DeclareMathOperator{\diam}{diam}

\DeclareMathOperator{\e}{e}
\DeclareMathOperator{\clos}{clos}

\DeclareMathOperator{\Leb}{Leb}             

\newtheoremstyle{colorplain}%
{\topsep}   
{\topsep}   
{\itshape}  
{0pt}       
{} 
{.}         
{5pt plus 1pt minus 1pt} 
{\textbf{\textcolor{RoyalBlue}{\textbf{\thmname{#1} \thmnumber{#2}}}}\thmnote{ (#3)}}
{}

\newtheoremstyle{colorremark}%
{\topsep}   
{\topsep}   
{}  
{0pt}       
{\itshape} 
{.}         
{5pt plus 1pt minus 1pt} 
{\textcolor{RoyalBlue}{\thmname{#1} \thmnumber{#2}}\thmnote{ (#3)}}
{}

\newtheoremstyle{colordefinition}%
{\topsep}   
{\topsep}   
{}  
{0pt}       
{} 
{.}         
{5pt plus 1pt minus 1pt} 
{\textcolor{RoyalBlue}{\textbf{\thmname{#1} \thmnumber{#2}}}\thmnote{ (#3)}}
{}

\theoremstyle{colorplain}
\newtheorem{theorem}{Theorem}

\numberwithin{theorem}{section}

\newtheorem{maintheorem}{Theorem}

\newtheorem{remark}[theorem]{Remark}
\newtheorem{example}[theorem]{Example}

\newtheorem{lemma}[theorem]{Lemma}

\newtheorem{proposition}[theorem]{Proposition}

\newtheorem{corollary}[theorem]{Corollary}

\theoremstyle{colorremark}

\newtheorem{question}{Question}

\makeatletter 
\newenvironment{proofabstract}[1][\proofname]{
  \par
  \pushQED{\qed}%
  \normalfont \topsep6\p@\@plus6\p@\relax
  \trivlist
  \item\relax
  {\itshape
    #1\@addpunct{.}}\hspace\labelsep\ignorespaces
}{%
  \popQED\endtrivlist\@endpefalse
}

\renewenvironment{proof}[1][Proof]{
  \setcounter{claim}{0}
  \setcounter{claimproof}{0}
  \par
  \pushQED{\qed}%
  \normalfont\topsep6\p@\@plus6\p@\relax
  \trivlist
  \item\relax
  {\itshape\color{RoyalBlue}#1\@addpunct{.}}\hspace\labelsep\ignorespaces
}{%
  \popQED\endtrivlist\@endpefalse
}

\makeatother

\newcounter{claimproof} 

{\begin{proofabstract}[Sketch proof]}
    {\end{proofabstract}}

\theoremstyle{colordefinition}
\newtheorem{definition}[theorem]{Definition}

\title{Partial section II: classification \\for general flows}
\author{Théo Marty}
\date{}

\begin{document}

\maketitle
\begin{abstract}
	This is the second article in a series that aims at classifying partial sections of flows, that is a general family of transverse surfaces.
	In this part, we classify partial cross-sections for all continuous flows, in the spirit of Schwartzman-Fried-Sullivan theory.

	We give a dynamical criterion for the existence of partial cross-sections, which is a direct consequence of part I of the series. Then we describe all partial cross-sections using a cohomological criterion, resembling Fried's criterion. We also characterize the cardinality of the set of partial cross-sections in a given cohomology class.
\end{abstract}

\section*{Introduction}
\addcontentsline{toc}{section}{Introduction}

This is the second paper in a series that aims at giving a global picture on the existence of partial cross-sections, the first being \cite{martyPS1}.

Surfaces transverse to flows received renewed attention in recent years, to help characterize dynamical and topological properties of flows. Transverse surfaces exist with several flavors, the most well-understood are the \emph{global cross-sections}: a compact hypersurface transverse to the flow that intersects every flow line. Fried\footnote{Schwartzman \cite{Schwartzman57} and Sullivan \cite{Sullivan76} contributed to the classification too.} \cite{Fried1982} classified the set of global cross-sections, up to isotopy along the flow, using a cohomological criterion. 
When we allow the surface to not necessarily intersect every flow line, partial results have been shown in restricted contexts, mostly for Anosov and pseudo-Anosov flows (see for instance \cite{Mosher1989, Mosher1990,Landry24}).

Let $M$ be a compact manifold and $\varphi$ be a continuous flow on $M$. A \emph{partial cross-section} is a compact hypersurface $S\subset M$ with $\partial S\subset\partial M$ which is transverse to the flow. 
Every partial cross-section admits a cohomology class~$\alpha$ in $H^1(M,\ZZ)$. For global cross-sections, Fried defined a set $D_\varphi\subset H_1(M,\RR)$, called the set of asymptotic directions, that satisfies the following. A class~$\alpha$ represents a global cross-section if and only if it satisfies the homological criterion $\alpha(D_\varphi)>0$. Additionally, two global cross-sections are cohomologous if and only if they are isotopic along the flow.
The case of partial cross-sections is more delicate. 
We are interested in the two following question. 


\begin{question}
	How can we characterize the set of partial cross-sections cohomologous to~$\alpha$, up to isotopy along the flow?
\end{question}

In the present paper, we completely answer the question. Note that the case $\alpha=0$ need a specific treatment. Our results often distinguish when~$\alpha$ is zero or not. For instance, in the case~$\alpha=0$, the existence of partial cross-sections is a direct application of Conley's theory. 
When~$\alpha$ is not zero, we developed the adapted theory in the first paper \cite{martyPS1} of the series. 
Say that~$\alpha$ is \emph{quasi-Lyapunov} if, roughly speaking, no pseudo-orbit $\gamma$ with small enough jump satisfies $\alpha([\gamma])<0$. 

\begin{maintheorem}\label{mainthm-existence}
	Let $M$ be a compact connected manifold, $\varphi$ be a continuous flow on $M$ and~$\alpha$ be in $H^1(M,\ZZ)$. There exists a partial cross-section cohomologous to~$\alpha$ if and only if either we have $\alpha\neq 0$ and $-\alpha$ is quasi-Lyapunov, or if we have $\alpha=0$ and $\varphi$ is not chain-recurrent.
\end{maintheorem}

When the condition in the theorem is satisfied, we have $\alpha(D_\varphi)\geq 0$, which is similar to Fried's criterion. Recall from the first paper that in general, $\alpha(D_\varphi)\geq 0$ does not imply that $-\alpha$ is quasi-Lyapunov. 

Let~$\alpha$ be as in the theorem. The partial cross-sections cohomologous to~$\alpha$ are in correspondence with a set of homology class relative to a recurrent-like set. Denote by~$\Rec_\alpha$ the \emph{$\alpha$-recurrent set} introduced in the first paper. For now, it can be though of as:~$\Rec_\alpha$ is the set of points that is contained in no partial cross-section cohomologous to~$\alpha$. 

A partial cross-section cohomologous to~$\alpha$ is disjoint from~$\Rec_\alpha$, so it induces a cohomologous class in $H^1(M,\Rec_\alpha,\ZZ)$. Note that this cohomology module behaves quite poorly since~$\Rec_\alpha$ can be any compact subset of $M$. 
We consider a slightly different cohomology module, which we denote by $H^1(M,\germ(\Rec_\alpha),\ZZ)$, called \emph{the cohomology relative to the germ} of~$\Rec_\alpha$. It is defined as a direct limit of cohomology modules, and as such, it has a natural topology. 
Note that there is a natural map $H^1(M,\germ(\Rec_\alpha),\ZZ)\xrightarrow{\pi_\germ} H^1(M,\ZZ)$.
We introduce similarly its continuous dual $H_1(M,\germ(\Rec_\alpha),\RR)$. 

\begin{maintheorem}\label{mainthm-classification}
	Under the conditions in Theorem~\ref{mainthm-existence}, there exists a subset $D_{\varphi,\alpha}$ of $H_1(M,\germ(\Rec_\alpha),\ZZ)$ that satisfies:
	\begin{itemize}
		\item any partial cross-section cohomologous to~$\alpha$ admits a cohomology class in $H^1(M,\germ(\Rec_\alpha),\ZZ)$,
		\item a non-zero element~$\beta$ in $H^1(M,\germ(\Rec_\alpha),\ZZ)$ represents a partial cross-section if and only if $\pi_\germ(\beta)=\alpha$ and $\beta(D_{\varphi,\alpha})\geq 0$ hold true,
		\item two partial cross-sections cohomologous in $H^1(M,\germ(\Rec_\alpha),\ZZ)$ are isotopic along the flow.
	\end{itemize}
\end{maintheorem}

We will restate the theorem as Theorem~\ref{thm-ps-classification-3}.
Denote by $\PS_\varphi(\alpha)$ the set of partial cross-section cohomologous to~$\alpha$, up to isotopy along the flow.
In order to prove Theorem~\ref{mainthm-classification}, we give two intermediate descriptions of $\PS_\varphi(\alpha)$ in Theorems~\ref{thm-ps-classification-1} and~\ref{thm-preL-to-order}. They are more combinatorial in nature, and they allow us to characterize the cardinality of $\PS_\varphi(\alpha)$.
Let us informally state our results.

\begin{maintheorem}[Concatenation of Theorems~\ref{thm-ps-countable},~\ref{thm-finite-pa} and~\ref{thm-unique-ps}]\label{mainthm-cardinal}
	Let~$\alpha$ be in $H_1(M,\ZZ)$ so that $\PS_\varphi(\alpha)$ is not empty.
	The set $\PS_\varphi(\alpha)$ is at most countable. In the case $\alpha=0$, $\PS_\varphi(0)$ is countable (infinite).

	Assume $\alpha\neq 0$. 
	Then $\PS_\varphi(\alpha)$ is finite if and only if some oriented graph, built from the connected components of~$\Rec_\alpha$, is finite and transitive.	
	Additionally, $\PS_\varphi(\alpha)$ is a singleton if and only if~$\Rec_\alpha$ is either empty or connected.
\end{maintheorem}

The article is organized as follows. 
In Section~\ref{sec-preliminary-top}, we show some elementary properties of partial cross-sections, and recall Fried's classification. 
In Section~\ref{sec-preliminary-dyn}, we recall some notation and results from the first paper in the series. All the dynamical objects are presented there.

In Section~\ref{sec-ps-classification}, we characterize the set of partial cross-sections. 
We first prove Theorem~\ref{mainthm-existence} using the following ideas.
Given a quasi-Lyapunov class~$\alpha$ in $ H_1(M,\ZZ)\setminus\{0\}$, it admits an~$\alpha$-equivariant and Lyapunov map $f\colon\wh M_\alpha\to\RR$. When a level set $f^{-1}(\{t\})$ is disjoint from the recurrent set, it projects onto $M$ into a partial cross-section cohomologous to $-\alpha$. Conversely, when~$S$ is a partial cross-section cohomologous to $-\alpha$, it can be used to build a function $f$ as above. We establish a deeper relation between partial cross-sections and equivariant Lyapunov maps, in Sections~\ref{sec-ps-subL-classification} and~\ref{sec-conley-order}.

 In Section~\ref{sec-card-PS} we characterize the cardinality of $\PS_\varphi(\alpha)$ and prove Theorem~\ref{mainthm-cardinal}. In particular, we characterize when $\PS_\varphi(\alpha)$ has cardinality one. More generally, characterizing when $\PS_\varphi(\alpha)$ has a given finite cardinality is possible, but it goes beyond the scope of this article. However, we give the sufficient tools for this characterization. 

In Section~\ref{sec-germ-hom}, we introduce the homology relative to the germ of a compact subset. This construction is quite general and not specific to the~$\alpha$-recurrent set. It may already be known to some degree by some experts. But we could not find it in the bibliography.
 
In Section~\ref{sec-hom-classification}, we build the set $D_{\varphi,\alpha}$ in $H_1(M,\germ(\Rec_\alpha),\RR)$ and prove Theorem~\ref{mainthm-classification}. We end with a discussion about the necessity of one hypothesis in that theorem.

We end our paper with three appendixes.
In Appendix~\ref{sec-smoothing}, we discuss the difference between continuous and smooth partial cross-sections. We give a quite weak hypothesis on the regularity of the flow under which every partial cross-section can be smoothed.

In Appendix~\ref{sec-Fried-sum}, we briefly discuss the relation with Theorem~\ref{mainthm-classification} and the so called Fried desingularization.  

In many applications, the germ of~$\Rec_\alpha$ is simpler than in the general case. For instance, when the flow has enough hyperbolicity,~$\Rec_\alpha$ should have finitely many connected components. 
We give in Appendix~\ref{app-practical-applications} a practical theorem to compute the homology and cohomology relative to the germ of~$\Rec_\alpha$, when~$\Rec_\alpha$ has finitely many connected components.

\vline

Let us illustrate the two main theorems with two examples. In Figure~\ref{fig-ps-existence}, we present two flows on the torus, $\varphi_1$ on the left and $\varphi_2$ on the right. 
We consider the two classes $\alpha=dx$ and $\alpha=-dx$.

The difference between $\alpha(D_\varphi)\geq0$ and $-\alpha$ being quasi-Lyapunov is illustrated here. 
For the two flows, we have $D_{\varphi_i}=\{(0,1)\}$. So both $dx$ and $-dx$ satisfy $\alpha(D_{\varphi_i})\geq0$. 
Furthermore, $-dx$ is quasi-Lyapunov for the two flows, but $dx$ is quasi-Lyapunov only for the flow $\varphi_1$. 
As a consequence, $\PS_{\varphi_1}(dx)$, $\PS_{\varphi_2}(dx)$ and $\PS_{\varphi_1}(-dx)$ are not empty, but $\PS_{\varphi_2}(-dx)$ is empty.

\begin{figure}[h]
	\begin{center}
		\begin{picture}(80,35)(0,0)
		\put(0,0){\includegraphics[width=80mm]{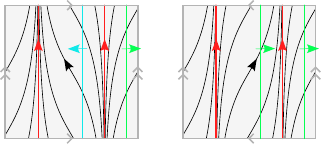}}
		\put(-4,20){$\varphi_1$}
		\put(80.5,20){$\varphi_2$}
		\color{red}
		\put(7.7,-2.5){$\Rec_\alpha$}
		\end{picture}
	\end{center}
	\caption{Existence of partial cross-sections, in green and blue, cohomologous to $dx$ and $-dx$.}
	\label{fig-ps-existence}
\end{figure}
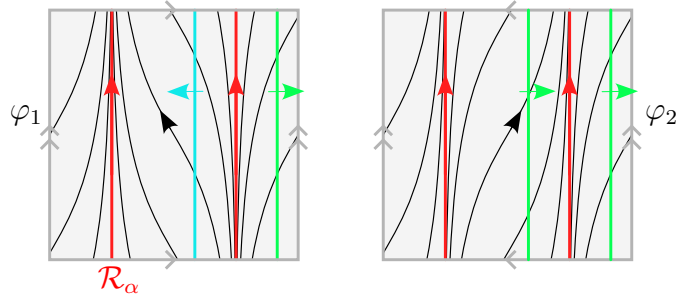

We end the introduction with Figure~\ref{fig-ps-classification} which illustrates the classification of partial cross-sections. Take some~$\alpha$ in $H^1(M,\ZZ)$ so that $-\alpha$ is quasi-Lyapunov. From the first paper, there exists a map $\wh f\colon M\to\bfrac{\RR}{\ZZ}$, cohomologous to~$\alpha$, so that $-f$ satisfies a Lyapunov property. That is, $f$ is constant on each connected component of~$\Rec_\alpha$ (in red), and it is cyclically increasing along the flow outside~$\Rec_\alpha$. 

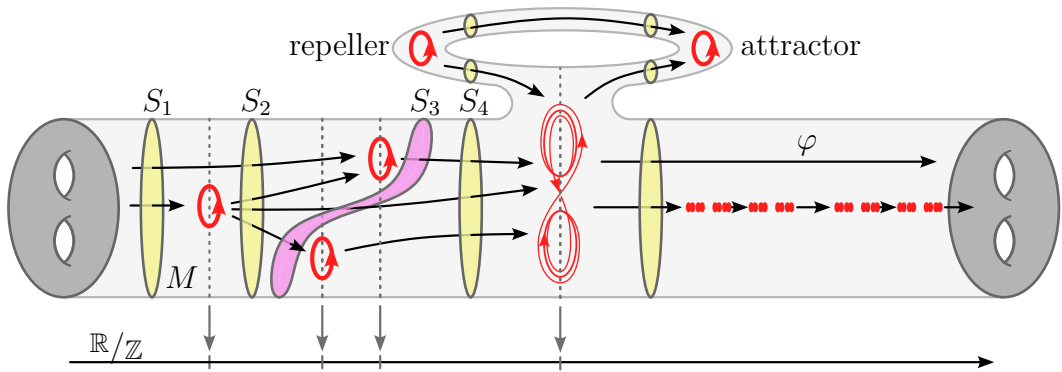
\begin{figure}[h]
	\begin{center}
		\begin{picture}(140,48)(0,0)
		\put(0,0){\includegraphics[width=140mm]{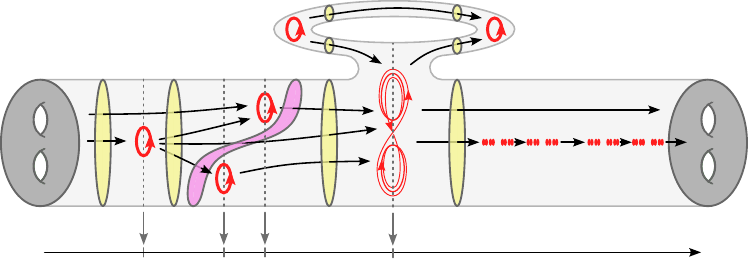}}
		\put(105, 29.5){$\varphi$}
		\put(18,34.6){$S_1$}
		\put(31,34.6){$S_2$}
		\put(53.5,34.6){$S_3$}
		\put(60,34.6){$S_4$}
		\put(21,11){$M$}
		\put(11,2){$\bfrac{\RR}{\ZZ}$}
		\put(97.5,42){attractor}
		\put(37.5,42){repeller}
		\end{picture}
	\end{center}
	\caption{Classification of partial cross-sections.}
	\label{fig-ps-classification}
\end{figure}

Consider a level set of $f$ that is disjoint from~$\Rec_\alpha$. By definition, it is transverse to the flow. So it is a partial cross-section cohomologous to~$\alpha$.
Four partial cross-sections obtained that way are represented in yellow. 
Note that they are separated by the~$\alpha$-recurrent set. So they are not isotopic to each other along the flow, even if they may be isotopic through regular embedded surfaces.
The partial cross-sections cohomologous to~$\alpha$ need not all have the same topology. For instance $S_1$ is connected but $S_4$ is not.

Every partial cross-section cohomologous to~$\alpha$ is obtained as a level set, for possibly a different function than $f$. For instance $S_3$ is not isotopic along the flow to a level set of $f$, so one needs to change $f$ to express $S_3$ as a level set. In particular, here, it is not possible to obtain all partial cross-sections cohomologous to $\alpha$ as the level sets of a single function $f$.

\paragraph{Acknowledgments.} I thank Pierre Dehornoy for the discussions that motivated this project. The research was conducted during my stay at the Université de Bourgogne, in Dijon, funded in honor of Marco Brunella.

{
	\hypersetup{linkcolor=black}
	\tableofcontents \label{ToC}
}

\section{Preliminary: the topology part} \label{sec-preliminary-top}

In the article,~$M$ is a topological manifold, connected, compact, possibly with boundary, and equipped with a compatible metric. Most of our arguments are valid without assuming much regularity on~$M$ nor $\varphi$. 

We fix a continuous flow $\varphi_t\colon M\to M$. All flows are assumed complete, that is defined at all time. In particular, $\varphi$ preserves $\partial M$. We allow~$M$ to have a non-empty boundary, so that the theory can be applied to manifolds obtained after blowing-up finitely many closed curves.


Since this is the second paper in a series, we introduce some terminology with less details, and refer to the first paper \cite{martyPS1} for more details on the preliminary.

\subsection{$\ZZ$-covering}\label{sec-Z-covering}

Fried \cite{Fried1983} showed the strong connection between global cross-sections, cohomology classes (of rank 1) and $\ZZ$-covers.

\paragraph{Descripiton of the rank 1 cohomology.} The cohomology set $H^1(M,\ZZ)$ plays a central role in this article. 

Let us describe that class in general. It is well-known that $H^1(M,\ZZ)$ is isomorphic (as a group) to the set $[M,\bfrac{\RR}{\ZZ}]$ of homotopy classes of continuous maps from $M$ to $\bfrac{\RR}{\ZZ}$. For instance, when $f\colon M\to\bfrac{\RR}{\ZZ}$ is smooth, its differential $df=f^*dt$ is a closed 1-form on $M$, so it induces a cohomology class in $H^1(M,\ZZ)$. It has integer coefficients since $df(\gamma)$ lies in $\ZZ$ for any oriented closed curve $\gamma$. 
We call the \emph{cohomology} class of $f\colon M\to\bfrac{\RR}{\ZZ}$ the element in $H^1(M,\ZZ)$ that corresponds to the homotopy class of $f$. 

Let $S\subset M$ be a compact, co-oriented hypersurface of $M$, with $\partial S\subset\partial M$. Similarly to above,~$S$ induces a cohomology class $[S]$ in $H^1(M,\ZZ)$ defined as follows. By assumption, there exists a continuous map $f\colon M\to\bfrac{\RR}{\ZZ}$ with $f^{-1}(\{0\})=S$, and so that the co-orientation on~$S$ goes from $f^{-1}(]\epsilon,0])$ to $f^{-1}([0,\epsilon[)$. Then the cohomology class $[S]$ in $H^1(M,\ZZ)$ is equal to the cohomology class of $f$. Any other possible choice of function $f$ is homotopic to $f$, so its cohomology class depends only on~$S$. 

\begin{remark}
	When $M$ is orientable and $d=\dim(M)$, the module $H^1(M,\ZZ)$ is naturally isomorphic to $H_{d-1}(M,\partial M,\ZZ)$, from a Poincaré duality. Given a compact and co-oriented hypersurface~$S$ of $M$, with $\partial S\subset\partial M$, it may be natural to consider its homology class in $H_{d-1}(M,\partial M,\ZZ)$ instead of $H^1(M,\ZZ)$. 
	But fundamentally, $H_{d-1}(M,\partial M,\ZZ)$ represents oriented hypersurfaces, and $H^1(M,\ZZ)$ represents co-oriented hypersurfaces, which is the case of~$S$. In particular, when $M$ is not orientable,~$S$ may be not representable by an element in $H_{d-1}(M,\partial M,\ZZ)$.
\end{remark}

\paragraph{$\ZZ$-covering}
Take a continuous map $f\colon M\to\bfrac{\RR}{\ZZ}$, assumed to be not null-cohomologous. We denote temporally by $M_f$ the connected manifold:
$$\wh M_f=\{(x,t)\in M\times\RR, f(x)\equiv t\smod1\},$$
equip with the covering map $\pi\colon\wh M_f\to M$ given by $\pi(x,t)=x$. We also denote by $\wh f\colon\wh M_f\to M$ the map $\wh f(x,t)=t$. 
Take another map $g\colon M\to\bfrac{\RR}{\ZZ}$ homotopic to $f$. Then there exists a map $h\colon M\to\RR$ which satisfies $h\equiv f-g\smod1$. And the map $H\colon\wh M_g\to\wh M_f$ defined by $H(x,t)=(x,t+h(t))$ is an isomorphism of covering over~$M$. So up to isomorphism, $\wh M_f$ depends only on the cohomology class of $f$.

Notice that $f$ is connected if and only if its cohomology class is primitive (not the multiple of another element). When it is the case, then $\wh M_f\to M$ is a $\ZZ$-covering map.

Take a non-zero class~$\alpha$ in $H^1(M,\ZZ)$. We chose an arbitrary continuous map $f\colon M\to\bfrac{\RR}{\ZZ}$ cohomologous to~$\alpha$. Then we denote by $\wh M_\alpha$ one connected component of $\wh M_f$, and we denote by $\pi_\alpha\colon\wh M_\alpha\to M$ the $\ZZ$-covering map discuss above.
A less explicit but more general description of $\wh M_\alpha$ is as the covering space over $M$ given by the kernel of the following composition of morphisms: 
$$\pi_1(M)\to H_1(M,\ZZ)\xrightarrow{\delta\mapsto\alpha(\delta)}\ZZ.$$
We used this definition in the first paper of the article, since we had to use more general Abelian covering.

Note that $\varphi$ lifts to a flow on $\wh M_\alpha$. From now on, we implicitly speak about the lifted flow when considering $\wh M_\alpha$.

\paragraph{Equivariance.} 
We denote by $k\cdot x$, for $x$ in $\wh M_\alpha$ and $k$ in $\ZZ$, the image of $x$ under the $\ZZ$-action of $n$.
A continuous map $f\colon\wh M_\alpha\to\RR$ is said \emph{$\ZZ$-equivariant} if for any $x$ and $k$ as above, we have $f(k\cdot x)=f(x)+k$. It is actually not the equivalence property that we are interested in, since it does not take in account~$\alpha$. 

When~$\alpha$ is not primitive, it can be uniquely written $\alpha=n_\alpha\beta$ for some $n_\alpha\geq 1$ and some primitive class~$\beta$ in $H^1(M,\ZZ)$. 
We say that $f\colon\wh M_\alpha\to\RR$ is \emph{$\alpha$-equivariant} if $\frac{1}{n_\alpha}f$ is $\ZZ$-equivariant, or equivalently, $f(k\cdot x)=f(x)+kn_\alpha$ holds for any $x,k$ as above.
Note that an~$\alpha$-equivariant map $f$ induces a continuous map $g\colon M\to\bfrac{\RR}{\ZZ}$ cohomologous to~$\alpha$, and with $f\equiv g\circ\pi_\alpha\smod1$ for any $x$ in $\wh M_\alpha$. Conversely, any continuous map $g\colon M\to\bfrac{\RR}{\ZZ}$ cohomologous to~$\alpha$ lifts to a continuous and~$\alpha$-equivariant map $f\colon\wh M_\alpha\to\RR$ that satisfies $f\equiv g\circ\pi_\alpha\smod1$.

\paragraph*{Choice of metric.} The soon introduced notion of $\epsilon$-pseudo-orbit depends on a choice of metric. 

Let $N$ be a topological space, and $d,d'$ be two metrics on $N$ compatible with its equipped topology. We say that $d$ and $d'$ are \emph{comparable at small scale} if for all $\epsilon>0$, there exists $\epsilon'>0$ so that every $\epsilon'$-ball for one metric is included in an~$\epsilon$-ball for the other metric. When $N$ is compact, any two compatible metrics are comparable at small scale. 

For the compact manifold~$M$, choose any metric compatible with the topology, denoted here by $d_M$. 
Given a non-zero~$\alpha$ in $H^1(M,\ZZ)$, 
When $N$ is given by a covering space $N\xrightarrow{\pi}M$, we lift the metric from~$M$ to $N$ as follows. 
We call the \emph{radius of injectivity} of~$M$, denote by $\diam(M)$, the supremum of the $\eta>0$ that satisfies that all balls of radius $\eta$ belong in a contractible open subset of~$M$. We choose a metric $d_N$ on $N$ which is invariant by the deck transformation of $\pi$, and so that for any ball $B\subset N$ of radius less than $\rad(M)$, the projection $B\xrightarrow{\pi}M$ is an isometry onto its image. Such a metric can be build for instance by taking chains of points in $N$, at distance less than $\rad(M)$, and summing the distance between successive points. Any two such metrics are comparable at small scale.

\subsection{Partial cross-section}\label{sec-partial-section}

Let~$S$ be a topologically sub-manifold $S\subset M$ of co-dimension one. We say that~$S$ is \emph{topologically transverse} to $\varphi$ if for all $x$ in~$S$, there exist $t>0$ and a neighborhood~$U$ of $x$ in~$S$, so that for all $y$ in $ U$ and $s$ in $[-t,t]$, if $\varphi_s(y)$ belongs to~$U$ then $s=0$ holds. 

A \emph{partial cross-section} of $\varphi$ is a (non-empty) topologically sub-manifold~$S$ inside $M$, of co-dimension one, compact, topologically transverse to $\varphi$, and with $\partial S\subset\partial M$.
A \emph{global cross-section} is a partial cross-section which intersects every flow lines of $\varphi$. We warn the reader that these notions have many names in the literature (Poincaré section, Birkhoff section, global section).

\begin{lemma}\label{lem-loc-flat}
	Any partial cross-section is locally flat.
\end{lemma}

\begin{proof}
	Let~$S$ be a partial cross-section of $\varphi$. By compactness, there exists $t>0$ which satisfies that for all $x$ in $M$, there exists at most one $s$ in $]-t,t[$ for which $\varphi_s(x)$ belongs to~$S$. Denote by $f\colon S\times]-t,t[\to M$ the map defined by $f(x,s)=\varphi_s(x)$. It is continuous, injective, and between two manifolds of the same dimension, so it is open. It implies that $\im(f)$ is a neighborhood of~$S$, and that $f$ is an embedding. Thus,~$S$ is locally flat.
\end{proof}

Given a partial cross-section~$S$ and a homotopy $\theta_t\colon S\to\RR$, for $t$ in $[0,1]$ and with $\theta_0=0$, denote by $S_{\theta_t}$ the set of $\varphi_{\theta_t(x)}(x)$ where $x$ varies in~$S$. We view $S_{\theta_t}$ as a homotopy of~$S$. When $S_{\theta_t}$ is embedded for all $t$, $S_{\theta_t}$ is said to be an \emph{isotopy along the flow}. 

\begin{lemma}
	Any isotopy among partial cross-sections can be parametrized as an isotopy along the flow.
\end{lemma}

Note that two partial cross-sections can be isotopic through embedded sub-manifolds, with $\partial S\subset \partial M$, without being isotopic through partial cross-sections. 
For instance a usual isotopy can cross a fixed point of the flow, which an isotopy along the flow cannot.
So we keep using the term ``isotopy along the flow'' to lift any ambiguity.

\begin{proof}
	Let~$S$ be a partial cross-section of $M$, and $h\colon[0,1]\times S\to M$ be an isotopy of~$S$.
	We first assume that $\partial S=\emptyset$ holds. 
	We prove that for some $s$ in $]0,1]$, the isotopy $h_t$ for $t$ in $[0,s]$ is an isotopy along the flow. The result follows from compactness.

	Similarly to the proof of Lemma~\ref{lem-loc-flat}, take $d>0$ and $f\colon S\times]-d,d[\to M$ a parametrization of a neighborhood of~$S$. Denote by $U=\im(f)$ and $\pi\colon U\to S$ the map that satisfies $\pi\circ f(x,t)=x$ for all $x$ in $ M$ and $|t|<d$. When $t$ is small enough, $S_t$ remains inside~$U$, so $\pi\circ h_t$ induces a map from~$S$ to itself. When $t$ is small enough, $\pi\circ h_t$ is additionally close to the identity map.

	We prove that $\pi\circ h_t\colon S\to S$ is surjective.
	Cover~$S$ with finitely many small open balls $B_1\cdots B_n$, and for all $i$, take a disc $D_i$ that contains the closure of $B_i$ in its interior. Given $x$ in~$S$, take an index $i_x$ so that $x$ belongs to $B_{i_x}$. When $t$ is small enough, $\pi\circ h_t(D_{i_x})$ still contains $x$ in its interior, since its boundary loop once around $x$. By compactness of~$S$, for every small enough $t$ and every $x$ in~$S$, $x$ lies in $\pi\circ h_t(D_{i_x})$. So $\pi\circ h_t$ is surjective.

	Since $h_t(S)$ is topologically transverse to the flow, $\pi\circ h_t$ is locally injective. Therefore, $\pi\circ h_t$ is a cover from~$S$ to itself. Note that the induced group morphism on $\pi_1(S)$ is constant in $t$, so it is the identity. Hence, $\pi\circ h_t$ is a cover of degree one, and Hence, a homeomorphism. 
	
	For any small $t$ and $x$ in~$S$, we write $h_t(x)=f(y,\theta(t,x))$ for a unique $\theta(t,x)$ in $ ]-d,d[$ and some $y$ in $ S$. It follows from above that $\theta$ is a homotopy, and that $S_{\theta_t}$ is an isotopy along the flow that parametrize the isotopy $S_t$.
	
	When $\partial S$ is not empty, we glue two copies $M\times\{0\}$ and $M\times\{1\}$ of $M$ along their boundary, using the map $\partial M\times\{0\}\to\partial M\times\{1\}$, that sends $(x,0)$ to $(x,1)$. The flow $\varphi$ lifts to a flow on the doubling of $M$. Then $S\times\{0,1\}$ is a partial cross-section on the new manifold, on which we can apply the previous case. The isotopy along the flow of the doubling of~$S$, in restriction to one copy of $M$, yields an isotopy along the flow of~$S$.
\end{proof}

Given a partial cross-section~$S$, we denote by $[S]$ its cohomology class in $H^1(M,\ZZ)$ (see Section~\ref{sec-Z-covering} for the definition of~$[S]$). 
Given~$\alpha$ in $H^1(M,\ZZ)$, we denote by $\PS_\varphi(\alpha)$ the set of isotopy classes along the flow of partial cross-sections~$S$ with $[S]=\alpha$. 
Fried classified the set of global cross-sections using its cohomology class.

\begin{theorem} [Fried {\cite[Thm. C and D]{Fried1982}}]\label{thm-Fried}
	Recall $M$ compact. The map $S\mapsto[S]$ yields a bijection between the set of global cross-sections~$S$ up to isotopy along the flow, and the set of cohomology class~$\alpha$ in $H^1(M,\ZZ)$ which satisfies $\alpha(D_\varphi)>0$.
\end{theorem}


A consequence is that when $\alpha(D_\varphi)>0$ holds, any partial cross-section cohomologous to~$\alpha$ is actually a global cross-section.

Fried stated his theorems with additional assumptions, knowing that his proofs work fine in a more general setting. For completeness, we will sketch a proof at the end of Section~\ref{sec-ps-existence}.

\section{Preliminary: the dynamic part} \label{sec-preliminary-dyn}

This section is a recall of the notation and of some results from the first paper in the series \cite{martyPS1}. 

\subsection{Conley's theory}
 
In this section, $N$ is either $M$ or $\wh M_\alpha$ for some non-zero class~$\alpha$ in $H^1(M,\ZZ)$.
Fix some $T>0$.
An $(\epsilon,T)$-\emph{pseudo-orbit} is a curve $\gamma\colon[0,l[\to N$, for some $l>0$, that satisfies the following:
\begin{itemize}
	\item $\gamma$ is right-continuous,
	\item $\gamma$ is continuous outside a finite set $\Delta$ that is disjoint from $[0,T[$,
	\item the points in $\Delta$ are distant by at least $T$,
	\item for any $t,s$ for which $\gamma$ is continuous on $[t,s]$, we have $\gamma(s)=\varphi_{s-t}\circ\gamma(t)$,
	\item for any $t$ in $\Delta$, the distance between $\gamma(t)$ and $\lim_{s\to t^-}\gamma(s)$ is smaller than $\epsilon$.
\end{itemize}

A \emph{periodic $\epsilon$-pseudo-orbit} is similarly a function $\gamma\colon\bfrac{\RR}{l\ZZ}\to M$, with $l>0$, that satisfies the same four conditions.
Note that any orbit arc, of positive length, is an $\epsilon$-pseudo-orbit. Similarly, any periodic orbit is a periodic $\epsilon$-pseudo-orbit.
From now on, we write $\gamma(u^-)=\lim_{t\to u^-}\gamma(t)$.
At a discontinuity time $u$ in $\Delta$, we say that $\gamma$ jumps from $\gamma(u^-)$ to $\gamma(u)$.
The scalar $l$ is called the \emph{length} of the pseudo-orbit, later denoted by $\len(\gamma)$. Additionally, we say that $\gamma$ goes from the point $\gamma(0)$ to the point $\gamma(\len(\gamma)^-)$.

\begin{definition}
	Let $x,y$ be two points in $N$. We write $x\recto y$ if for all $\epsilon>0$, there exists an $\epsilon$-pseudo-orbit from $x$ to $y$. 
	It does not depend on the value of $T>0$ (once fixed). 
\end{definition}

A point $x$ in $N$ is said \emph{recurrent} when we have $x\recto x$. The set of recurrent point is called the \emph{recurrent set}. We denote by $\Rec$ the recurrent set on the manifold~$M$. Given~$\alpha$ in $ H^1(M,\RR)$, we denote by $\wh\Rec_\alpha$ the recurrent set on~$\wh M_\alpha$. The relation $x\rectot y$, given by $x\recto y$ and $y\recto x$, is an equivalence relation. Its equivalence classes are called \emph{recurrence chain}. Similarly, for two recurrence chains $R,R'$, we denote by $R\recto R'$ if we have $x\recto y$ for some/all $x$ in $ R$ and $y$ in $ R'$. 

We will denote by $\wh\Rec_\alpha$ the recurrent set of the lifted flow on $\wh M_\alpha$. It is well-known that the recurrence chains satisfy:

\begin{lemma}[See \cite{martyPS1} Corollary 1.6 for instance]\label{lem-tot-disc}
	The set of recurrence chains in $\wh M_\alpha$, with the quotient topology, is totally disconnected.
\end{lemma}

A continuous map $f\colon N\to\RR$ is \emph{Lyapunov} if the three following properties are satisfied: $f$ is constant on each recurrence chain, two distinct recurrence chains have distinct values, and $f$ is decreasing along the flow outside the recurrent set. 

\begin{theorem}[\cite{Conley1978}]\label{thm-Conley-L}
	Let $M$ be a compact manifold. Any continuous flow on $M$ admits a Lyapunov function.
\end{theorem}

We consider a weaker version of the Lyapunov function, useful to control partial cross-sections.
A continuous map $f\colon N\to\RR$ is called \emph{pre-Lyapunov} if for any $x,y$ in $ N$, $x\recto y$ implies $f(x)\geq f(y)$. It implies that $f$ is constant on each recurrence chain.

\subsection{Asymptotic pseudo-directions}\label{sec-ass-ps-dir}

Assume $\epsilon>0$ smaller than half the injectivity radius of $M$.
For short, we will simply say that $\epsilon$ is \emph{small}.
Take an $(\epsilon, T)$-pseudo-orbit $\gamma$. We denote by $[\gamma]$ the homology class in $H^1(M,\ZZ)$ of the oriented curve obtained from $\gamma$ by connecting any jumps in $\gamma$ by curves that each remains in a ball of radius $\epsilon$. Since $\epsilon$ is small, the cohomology class of $\gamma$ does not depend on the choice of curves used to close $\gamma$.
We define the set
$$D_{\varphi,\epsilon,T}=\clos\left(\left\{\tfrac{1}{\len(\gamma)}[\gamma]\in H_1(M,\RR), \gamma\text{ a periodic $\epsilon$-pseudo-orbit}\right\}\right)$$ 
to be the closure of the set of elements $\tfrac{1}{\len(\gamma)}[\gamma]$. Note that $\tfrac{1}{\len(\gamma)}[\gamma]$ remains in a bounded region of $H_1(M,\RR)$, so $D_{\varphi,\epsilon,T}$ is compact.

\begin{definition}
	We define $D_\varphi=\bigcap_{\epsilon>0}D_{\varphi,\epsilon,T}$, which we call the \emph{set of asymptotic pseudo-directions} of $\varphi$.
\end{definition}

Notice that $D_\varphi$ does not depend on $T$. The definition differ from the set of asymptotic directions given by Fried \cite{Fried1982}. These two sets span the same convex set \cite[Appendix A]{martyPS1}. Our main results depend on the convex hull of $D_\varphi$ only, so we may switch $D_\varphi$ and Fried's set in them.

\subsection{Quasi-Lyapunov class}\label{sec-subL}

Let~$\alpha$ be in $H^1(M,\RR)\setminus\{0\}$ and $f\colon\wh M_\alpha\to\RR$ be an~$\alpha$-equivariant and continuous map. 
Given $\epsilon>0$ small and an $\epsilon$-pseudo-orbit $\gamma$ in $M$, we denote by $\int_\gamma df$ the following quantity. Lift $\gamma$ to an $\epsilon$-pseudo-orbit $\wh\gamma$ in $\wh M_\alpha$, that goes from a point $x$ to a point $y$, and set $\int_\gamma df=f(y)-f(x)$. Note that $df$ is an abuse of notation, we do not assume~$f$ differentiable. 

We say that $f$ is \emph{$C$-quasi-Lyapunov} if there exists $\epsilon>0$ so that for any $\epsilon$-pseudo-orbit $\gamma$ in $M$, we have $\int_\gamma df\leq C$.
We say that $f$ is \emph{quasi-Lyapunov} if it is $C$-quasi-Lyapunov for some~$C$. 
A cohomology class~$\alpha$ in $ H^1(M,\RR)$ is said \emph{quasi-Lyapunov} if any~$\alpha$-equivariant and continuous map $f\colon\wh M\to\RR$ is quasi-Lyapunov. Note that $\alpha=0$ is always quasi-Lyapunov.

\begin{lemma}[\cite{martyPS1} Lemma 3.6]\label{lem-part1-qL}
	A class $\alpha$ in $H^1(M,\ZZ)$ is quasi-Lyapunov, if and only if we have $\alpha(D_{\varphi,\epsilon,T}))\leq0$ for some/any $\epsilon>0$ small enough. In particular, it implies $\alpha(D_\varphi)\leq0$.
\end{lemma}

One of the main results from the part I is the existence of equivariant Lyapunov functions.

\begin{theorem}[\cite{martyPS1} Theorem 4.1]\label{thm-spectral-decomp}
	Assume $M$ compact and let $\alpha$ in $H^1(M,\ZZ)$ be non-zero. Then $\alpha$ is quasi-Lyapunov if and only if there exists an~$\alpha$-equivariant and Lyapunov map $f\colon\wh M_\alpha\to\RR$.
\end{theorem}

\subsection{The~$\alpha$-recurrent set}

Fix a non-zero cohomology class~$\alpha$ in $ H^1(M,\RR)$. We say that a point $x$ in $ M$ is \emph{$\alpha$-recurrent} if for every $\epsilon>0$ small and every $T>0$, there exist a periodic $\epsilon$-pseudo-orbit $\gamma$ that passes through $x$ and that satisfy $\alpha([\gamma])=0$. We denote by~$\Rec_\alpha$ the set of~$\alpha$-recurrent point, called \emph{$\alpha$-recurrent set}. Two~$\alpha$-reccurent points $x,y$ in $M$ are said to be~$\alpha$-equivalent if for all $\epsilon,T>0$ there exists a periodic $\epsilon$-pseudo-orbit $\gamma$ that passes through $x$ and $y$ and with $\alpha([\gamma])=0$. The~$\alpha$-equivalence relation is an equivalence relation. We call \emph{$\alpha$-recurrence chains} the~$\alpha$-equivalence classes. Note that the~$\alpha$-equivariance chains are invariant by the flow.

The relation between Fried's work and the $\alpha$-recurrent set is given by:

\begin{proposition}[\cite{martyPS1} Proposition 3.11]\label{prop-part1-empty-rec}
	We have $\alpha(D_\varphi)>0$ if and only $-\alpha$ is quasi-Lyapunov and $\Rec_\alpha$ is empty.
\end{proposition}

\section{Characterization of partial cross-sections}\label{sec-ps-classification}

The classification of global-cross-sections is given by a simple homological criterion.
The classification of partial cross-sections is more complicated than the one of global cross-sections. The existence and uniqueness do not hold true in general. We give here some elements for both directions.

\subsection{Existence of partial cross-section}\label{sec-ps-existence}

In this section, we focus on the existence and prove the three following theorems. The criterion is not the same for $\alpha=0$ and $\alpha\neq 0$.

\begin{theorem}\label{thm-ps-to-subL}
	Assume $M$ compact and connected.
	For any non-zero class~$\alpha$ in $H^1(M,\ZZ)$, we have $\PS_\varphi(\alpha)\neq\emptyset$ if and only if $-\alpha$ is quasi-Lyapunov.
\end{theorem}

The case $\alpha=0$ is different, and already known by experts. The flow $\varphi$ is said \emph{chain recurrent} if the recurrence set of $\varphi$ is equal to $M$. 

\begin{theorem}\label{thm-ps-to-non-rec}
	Assume $M$ compact and connected.
	We have $\PS_\varphi(0)\neq\emptyset$ if and only if the restriction of $\varphi$ is not chain recurrent.
\end{theorem}

We will also prove the following. 
Notice that for $\alpha=0$, the~$\alpha$-recurrent set is simply the recurrent set of~$\varphi$.

\begin{theorem}\label{thm-ps-rec-disjoint}
	Assume $M$ compact and connected.
	Let~$\alpha$ be in $H^1(M,\ZZ)$ with $\PS_\varphi(\alpha)\neq\emptyset$. Then the~$\alpha$-recurrent set~$\Rec_\alpha$ is exactly the set of points in $M$ that lie on no partial cross-section cohomologous to~$\alpha$.
\end{theorem} 

We first build partial a cross-section cohomologous to~$\alpha$. Then we clarify the relation with the~$\alpha$-recurrent set. In the case $\alpha=0$, we take the convention $\wh M_\alpha=M$ and any function $f\colon\wh M_\alpha\to\RR$ is~$\alpha$-equivariant.

\begin{lemma}\label{lem-pre-image-ps}
	Let~$\alpha$ be as in Theorem~\ref{thm-ps-rec-disjoint}, $f\colon\wh M_\alpha\to\RR$ be~$\alpha$-equivariant, $t$ in $f(\wh M_\alpha)$ and $S=f^{-1}(\{t\})$. If $-f$ is Lyapunov and $t$ is not in $f(\wh\Rec_\alpha)$, then the image of~$S$ in $M$ is a partial cross-section cohomologous to~$\alpha$. 
	
	If $f$ is only assumed to be decreasing along the flow on a neighborhood of~$S$, the same conclusion holds true.
\end{lemma}

\begin{proof}
	First note that when $-f$ is Lyapunov and $t$ is not in $f(\wh\Rec_\alpha)$, then~$S$ lies outside the recurrent set of the flow on $\wh M_\alpha$. Since~$S$ and $\wh\Rec_\alpha$ are compact, $\wh M_\alpha\setminus\wh\Rec_\alpha$ is a neighborhood of~$S$, and $f$ is decreasing along the flow on that set. Thus, it is enough to prove the lemma under the second assumption. 
	
	Since $f$ is increasing along the flow on a neighborhood of~$S$,~$S$ is a topological sub-manifold of codimension 1, and it is topologically transverse to the flow. It is non-empty by construction, compact by co-compactness of $\wh M_\alpha$, and satisfies $\partial S\subset\partial\wh M_\alpha$. 
	
	Since $f$ is~$\alpha$-equivariant, it induces an embedding from~$S$ to $M$. Indeed, if $\pi_\alpha(x)=\pi_\alpha(y)$ holds for some points $x,y$ in~$S$, then there exists $n$ in $\ZZ$ with $y=n\cdot x$. The~$\alpha$-equivariance and $f(x)=t=f(y)$ imply $n=0$ and $y=y$.
	
	Thus, the image of~$S$ in $M$ is a partial cross-section. It is cohomologous to~$\alpha$ by construction.
\end{proof}

Let~$S$ be a partial cross-section and $\alpha=[S]$. We build a function from $M$ to $\bfrac{\RR}{\ZZ}$ cohomologous to~$\alpha$. Since~$S$ is compact and transverse to the flow $\varphi$, there exists $\epsilon>0$ so that the flow box $\varphi_{[-\epsilon,\epsilon]}(S)$ is embedded. Define the map
$$F_{S}\colon M\to\bfrac{\RR}{\ZZ}$$
by $F_S(\varphi(t,x))= \frac{t+\epsilon}{2\epsilon}$ for any $t$ in $[-\epsilon,\epsilon]$ and $x$ in $ S$, and $F_S\equiv0$ outside $\varphi_{[-\epsilon,\epsilon]}(S)$. It is continuous, non-decreasing along the flow, and constant on the recurrent set. By construction, we have $S=F_S^{-1}(\tfrac{1}{2})$, so~$S$ and $F_S$ are cohomologous.

It will be convenient to consider pre-Lapunov map from $M$ to the circle. A map $f\colon M\to\bfrac{\RR}{\ZZ}$ is called \emph{pre-Lyapunov} if given $\alpha=[f]$ in $ H^1(M,\ZZ)$ its cohomology class, any continuous lift $\wh f\colon\wh M_\alpha\to\RR$ of $f$ is pre-Lyapunov.

It follows from $\wh F_S(\wh\Rec_\alpha)\subset\ZZ$ and \cite[Lemma 1.7]{martyPS1} that:

\begin{lemma}\label{lem-F-ps}
	Let~$S$ be a partial cross-section. The map $-F_{S}$ is pre-Lyapunov. Thus, using $\alpha=[S]$, $-\alpha$ is quasi-Lyapunov and $F_{S}(\Rec_\alpha)=\{0\}$ hold true.
\end{lemma}

We now prove the three theorems stated earlier. First the existence of partial cross-sections in the case $\alpha\neq 0$.

\begin{proof}[Proof of Theorem~\ref{thm-ps-to-subL}]
	Take a non-zero class~$\alpha$ in $H^1(M,\ZZ)$.

	Assume that $-\alpha$ is quasi-Lyapunov. Then there exists an $-\alpha$-equivariant Lyapunov map $g\colon\wh M_\alpha\to\RR$ (see Theorem~\ref{thm-spectral-decomp}). It follows from Lemma~\ref{lem-tot-disc} that $g(\wh\Rec_\alpha)$ has an empty interior. Note that $g$ is surjective, since it is continuous and~$\alpha$-equivariant on a connected set. So there exists $s$ in its image which does not lie in $g(\Rec_\alpha)$. Hence, by Lemma~\ref{lem-pre-image-ps}, the image in $M$ of $g^{-1}(\{s\})$ is a partial cross-section cohomologous to~$\alpha$.

	Assume now that there exists a partial cross-section~$S$ cohomologous to~$\alpha$. Denote by $\wh F_S\colon\wh M_\alpha\to\RR$ a lift of the map $F_S$ constructed above. It follows from above that $-\wh F_S$ is pre-Lyapunov. The map $-\wh F_S$ is $-\alpha$-equivariant, so $-\alpha$ is quasi-Lyapunov (see Theorem~\ref{thm-spectral-decomp}).
\end{proof}

Secondly we prove the existence of partial cross-sections in the case $\alpha=0$.

\begin{proof}[Proof of Theorem~\ref{thm-ps-to-non-rec}]
	Assume that $\varphi$ is not chain recurrent. Then it has at least two recurrence chains.
	It follows from Theorem~\ref{thm-Conley-L} that there exists a Lyapunov map $f\colon M\to\RR$ for $\varphi$. By assumption, $f$ admits distinct values on distinct recurrence chains, so $f$ is not constant. Since it is continuous on a connected set, $f(M)$ is a closed and non-trivial interval. It follows from Lemma~\ref{lem-tot-disc} that $f(\Rec_0)$ has empty interior inside $\RR$, so there exists a value $t$ in $f(M)$ outside $f(\Rec_0)$. It follows from Lemma~\ref{lem-pre-image-ps} that $f^{-1}(\{t\})$ is a null-cohomologous partial cross-section.

	Conversely, assume that there exists a null-cohomologous partial cross-section~$S$. It follows from Lemma~\ref{lem-F-ps} that~$S$ is disjoint from $\Rec_0$. Thus, $\varphi$ is not chain recurrent.
\end{proof}

End now prove that the~$\alpha$-recurrent set is the complementary of the partial cross-section cohomologous to~$\alpha$.

\begin{proof}[Proof of Theorem~\ref{thm-ps-rec-disjoint}]
	Let~$\alpha$ be in $H^1(M,\ZZ)$ with $\PS_\varphi(\alpha)\neq\emptyset$.
	According to Theorem~\ref{thm-spectral-decomp} in the case $\alpha\neq 0$, and to Theorem~\ref{thm-Conley-L} in the case $\alpha=0$, there exists an~$\alpha$-equivariant map $f\colon\wh M_\alpha\to\RR$ so that $-f$ is Lyapunov. It follows from Lemma~\ref{lem-tot-disc} that $f(\wh\Rec_\alpha)$ has an empty interior. 
	
	Take a point $p$ in $M$ that is not~$\alpha$-recurrent, and $\wh p$ a lift of $p$ in $\wh M_\alpha$. 
	Denote by $\wh\varphi^\alpha$ the lifted flow on $\wh M_\alpha$.
	Since $f$ is increasing along the flow on the orbit of $\wh p$, there exists $s$ in $\RR$ for which $t=f\circ\wh\varphi^\alpha_s(\wh p)$ does not lie in~$f(\wh\Rec_\alpha)$. The image of $f^{-1}(t)$ in $M$ is a partial cross-section cohomologous to~$\alpha$, according to Lemma~\ref{lem-pre-image-ps}. Denote it by~$S$. Then $\varphi_{-s}(S)$ is also a partial cross-section cohomologous to~$\alpha$, and it contains $p$. So any point outside~$\Rec_\alpha$ is contained in a partial cross-section cohomologous to~$\alpha$.

	We now prove the converse. Let~$S$ be a partial cross-section cohomologous to~$\alpha$, and $F_{S}\colon M\to\bfrac{\RR}{\ZZ}$ the map defined above. It follows from Lemma~\ref{lem-F-ps} that $F_{S}(\wh\Rec_\alpha)=\{0\}$. So $F_{S}(S)=\tfrac{1}{2}$ implies that~$S$ is disjoint from~$\Rec_\alpha$.
\end{proof}

Let us now sketch a proof of Fried's classification, in the general setting.

\begin{proof}[Sketch of proof of Theorem~\ref{thm-Fried} in the general settings]
	Take $\alpha\neq0$ and assume that $\alpha(D_\varphi)>0$ holds. From Proposition~\ref{prop-part1-empty-rec}, $-\alpha$ is quasi-Lyapunov and~$\Rec_\alpha$ is empty. By Theorem~\ref{thm-spectral-decomp}, there exists an~$\alpha$-equivariant map $f\colon\wh M_\alpha\to\RR$ so that $-f$ is Lyapunov. In particular $f$ is everywhere increasing along the flow.  Lemma~\ref{lem-pre-image-ps} implies that any level set~$S$ of $f$ is a partial cross-section cohomologous to~$\alpha$. And since $f$ is increasing along the flow, any orbit of the flow in $\wh M_\alpha$ intersects~$S$ in exactly one point. So~$S$ is a global cross-section. 
	
	Each step above has a converse. So if there exists a global cross-section cohomologous to~$\alpha$, then $\alpha(D_\varphi)>0$.

	Let~$\alpha$ and $f$ be as above, and take $S_1,S_2$ be two global cross-sections cohomologous to~$\alpha$, and $\wh S_i$ a lift of $S_i$ in $\wh M_\alpha$. As said above, any orbit in $\wh M_\alpha$ intersects exactly once $\wh S_i$. Following the flow from $\wh S_1$ to $\wh S_2$ induces an isotopy along the flow from $\wh S_1$ to $\wh S_2$. And its image in $M$ is an isotopy along the flow from $S_1$ to $S_2$.
\end{proof}

\subsection{Equivalence with pre-Lyapunov maps}\label{sec-ps-subL-classification}

We give a bijection between $\PS_\varphi(\alpha)$ and a set of homotopy classes of pre-Lyapunov maps.

Denote by 
$[(M,\Rec_\alpha)\col(\bfrac{\RR}{\ZZ},0)]_\alpha$
the set of maps $f\colon M\to\bfrac{\RR}{\ZZ}$, cohomologous to~$\alpha$, which satisfy $f(\Rec_\alpha)=0$ and so that $-f$ is pre-Lyapunov. We consider maps in $[(M,\Rec_\alpha)\col(\bfrac{\RR}{\ZZ},0)]_\alpha$ up to homotopy through that set.

In the case $\alpha=0$, we denote by $0$ the homotopy class of functions in $[(M,\Rec_\alpha)\col(\bfrac{\RR}{\ZZ},0)]_\alpha$ that contains the constant map~0. The set of non-zero homotopy classes inside $[(M,\Rec_\alpha)\col(\bfrac{\RR}{\ZZ},0)]_\alpha$ is denoted by \emph{$\pL(\alpha)$}. As discussed below, every element in $\pL(\alpha)$ corresponds to a partial cross-section cohomologous to $\alpha$. In the case $\alpha=0$, we remove the zero map since it corresponds to the empty set, which by convention is not a partial cross-section. 

Note that given a partial cross-section~$S$ cohomologous to~$\alpha$, the function $F_{S}$ (defined previously) lies in $[(M,\Rec_\alpha)\col(\bfrac{\RR}{\ZZ},0)]_\alpha$, which is not homotopic to a constant. The image of $F_S$ in $\pL(\alpha)$ does not depend on the choices made to build $F_S$, and it is invariant by isotopy along the flow on~$S$. So the map $S\mapsto F_{S}$ induces a map from $\PS_\varphi(\alpha)\to\pL(\alpha)$. We denote that map by $[F]$.

\begin{theorem}\label{thm-ps-classification-1}
	Assume $M$ compact and connected. For any~$\alpha$ in $H^1(M,\ZZ)$, the map 
	$\PS_\varphi(\alpha)\xrightarrow{[F]}\pL(\alpha)$
	is a bijection.
\end{theorem}

\begin{proof}
	Let us first prove that $[F]$ is surjective.
	Take $g$ in $[(M,\Rec_\alpha)\col(\bfrac{\RR}{\ZZ},0)]_\alpha$, a scalar $t$ in $\bfrac{\RR}{\ZZ}$ and some $\epsilon>0$. Up to replacing $g$ with a convolution along the flow between $g$ and a Gaussian map, we may assume that $g$ is increasing along the flow on $g^{-1}((\bfrac{\RR}{\ZZ})\setminus\{0\})$.
	We claim that $g$ is surjective. In the case $\alpha\neq 0$, it follows from the fact that $g$ is not null-cohomologous. In the case $\alpha=0$, it follows from the fact that $g$ is not constant.

	It follows from Lemma~\ref{lem-pre-image-ps} that $S=g^{-1}(\tfrac{1}{2})$ is a partial cross-section cohomologous to~$\alpha$. The maps $F_{S}$ and $g$ are cohomologous, and they coincide on $\Rec_\alpha\cup S$. So a linear interpolation between $F_{S}$ and $g$ (fixing $\Rec_\alpha\cup S$), yields a homotopy inside $[(M,\Rec_\alpha)\col(\bfrac{\RR}{\ZZ},0)]_\alpha$. Therefore, the image of $g$ and $F_{S}$ in $\pL(\alpha)$ are equal, and $[F]$ is surjective.

	We now prove that $[F]$ is injective. Let $S_0,S_1$ be two partial cross-sections cohomologous to~$\alpha$, and which satisfy that $F_{S_0}$ and $F_{S_2}$ are homotopic inside the set $[(M,\Rec_\alpha)\col(\bfrac{\RR}{\ZZ},0)]_\alpha$. We write $g_i=F_{S_i}$, and we take a homotopy $g_t$ for $t$ in $[0,1]$ from $F_{S_1}$ to $F_{S_2}$. Take a Gaussian map $h_a(t)=\sqrt{\tfrac{a}{\pi}}\e^{-at^2}$. Consider the homotopy $g_0*_\varphi h_a$ with $a$ that goes from $+\infty$ to $1$. That homotopy goes from $g_0$ to $g_0*_\varphi h_1$. We concatenate to it the homotopies $g_t*_\varphi h_1$ for $t$ in $[0,1]$ and $g_1*_\varphi h_a$ for $a$ in $[1,+\infty[$. It yields a homotopy inside $[(M,\Rec_\alpha)\col(\bfrac{\RR}{\ZZ},0)]_\alpha$ from $g_0$ to $g_1$. 
	
	We denote by $f_t$, with $t$ in $[0,1]$, a reparametrization of the homotopy described above. Note that for any $0<t<1$, $f_t$ is increasing along the flow inside $f_t^{-1}(\bfrac{\RR}{\ZZ}\setminus\{0\})$. It follows from Lemma~\ref{lem-pre-image-ps} that $S_t=f_t^{-1}(\tfrac{1}{2})$ is a partial cross-section cohomologous to~$\alpha$. Hence, $S_t$ is an isotopy of partial cross-sections from $S_0$ and $S_1$, and is Thus, an isotopy along the flow.
\end{proof}

\subsection{Equivalence with Conley's order}\label{sec-conley-order}

We give another bijection from $\PS_\varphi(\alpha)$, to a more combinatorial set.


Denote by $\wh\varphi^\alpha$ flow on $\wh M_\alpha$ obtained by lifting $\varphi$, and $\wh\Rec_\alpha^*$ the set of recurrence chains of $\wh\varphi^\alpha$, equipped with the quotient topology given by the projection $\wh\Rec_\alpha\to\wh\Rec_\alpha^*$. It is equipped with a (partial) order $\recto$ defined by: $R_1\recto R_2$ holds if we have $x_1\recto x_2$ for some/any $x_1$ in $ R_1$ and $x_2$ in $ R_2$. Note that the order is invariant under the action of $\ZZ$ over $\wh M_\alpha$. The order $(\wh\Rec_\alpha^*,\recto)$ is sometimes called Conley's order. 

\begin{definition}
	We define $\wh I_\alpha(\wh\Rec_\alpha^*,\ZZ)$ to be the set of~$\alpha$-equivariant, non-constant, non-decreasing and continuous (equivalently locally constant) maps from $(\wh\Rec_\alpha^*,\recto)$ to $(\ZZ,\leq)$. 

	We also define $I_\alpha(\wh\Rec_\alpha^*,\ZZ)$ to be the set of equivalence classes of functions in $\wh I_\alpha(\wh\Rec_\alpha^*,\ZZ)$ up to an additive constant. That is $f\simeq g$ if $f-g$ is a constant. 
\end{definition}

Note that when $\wh\Rec_\alpha^*$ is empty, by convention, $\wh I_\alpha(\wh\Rec_\alpha^*,\ZZ)$ is equal to the singleton with the unique function from $\emptyset$ to $\ZZ$, and $I_\alpha(\wh\Rec_\alpha^*,\ZZ)$ is a singleton too.

Take a map $f$ in $[(M,\Rec_\alpha)\col(\bfrac{\RR}{\ZZ},0)]_\alpha$. It lifts to an~$\alpha$-equivariant map from $\wh M_\alpha$ to $\RR$, whose opposite is pre-Lyapunov. It then restricts to a continuous map $\wh f^*\colon\wh\Rec_\alpha^*\to\ZZ$. The map $\wh f^*$ is non-decreasing since $-f$ is pre-Lyapunov. Thus, $\wh f^*$ belongs to $\wh I_\alpha(\wh\Rec_\alpha^*,\ZZ)$. 
We denote by $[\wh f^*]$ the equivalence class of $\wh f^*$ inside $I_\alpha(\wh\Rec_\alpha^*,\ZZ)$.

Let $f_s$, for $s$ in $[0,1]$, be a homotopy of maps in $[(M,\Rec_\alpha)\col(\bfrac{\RR}{\ZZ},0)]_\alpha$. We choose lifts $\wh f_s\colon\wh M_\alpha\to\RR$ of the maps $f_s$, continuously in $s$. Then by continuity, the restrictions $\wh f^*_s\colon\wh\Rec_\alpha^*\to\ZZ$ are all equal. 
Thus, $[\wh f^*]$ depends only on the homotopy class of $f$ in $\pL(\alpha)$. Denote by $\pL(\alpha)\xrightarrow{[\wh *]}I_\alpha(\wh\Rec_\alpha^*,\ZZ)$ the corresponding map.

\begin{theorem}\label{thm-preL-to-order}
	Assume $M$ compact and connected. 
	Let~$\alpha$ be in $H^1(M,\ZZ)$ so that $\PS_\varphi(\alpha)$ is not empty. 
	Then the map $\pL(\alpha)\xrightarrow{[\wh *]}I_\alpha(\wh\Rec_\alpha^*,\ZZ)$ is a bijection.
\end{theorem}

In this theorem, it is necessary to assume $\PS_\varphi(\alpha)\neq\emptyset$. The set $I_\alpha(\wh\Rec_\alpha^*,\ZZ)$ does not contain enough information to determine whether $\PS_\varphi(\alpha)$ is empty or not.

\begin{proof}
	Assume first that $\wh\Rec_\alpha^*$ is empty. Then~$\alpha$ is not zero, since the recurrent set on $M$ is not empty. Thus, $-\alpha$ is quasi-Lyapunov from Theorem~\ref{thm-ps-to-subL}. By Proposition~\ref{prop-part1-empty-rec}, we have $\alpha(D_\varphi)>0$. It follows from Theorem~\ref{thm-Fried}, that $\PS_\varphi(\alpha)$ is a singleton, and so is $I_\alpha(\wh\Rec_\alpha^*,\ZZ)$. Thus, $[\wh *]$ is bijective.

	We now assume $\wh\Rec_\alpha^*\neq\emptyset$.
	We first show the injectivity of the map $[\wh*]$.
	Take two functions $f,g$ in $[(M,\Rec_\alpha)\col(\bfrac{\RR}{\ZZ},0)]_\alpha$ and assume that $[\wh f^*]=[\wh g^*]$ holds. From the equality, we may choose two lifts $\wh f,\wh g\colon\wh M_\alpha\to\RR$ of $f$ and $g$, so that they coincide on the subset $\wh\Rec_\alpha$. 
	
	Define $\wh h_t=t\wh f+(1-t)\wh g$, for $t$ in $[0,1]$. It provides a homotopy from $\wh f$ to $\wh g$. 
	The map $\wh h_t$ projects to a map $h_t\colon M\to\bfrac{\RR}{\ZZ}$, which lies inside $[(M,\Rec_\alpha)\col(\bfrac{\RR}{\ZZ},0)]_\alpha$. We have $h_0=g$ and $h_1=f$, so the image of $f$ and $g$ inside $\pL(\alpha)$ are equal.

	The surjectivity is a direct consequence of \cite[Theorem 4.11]{martyPS1}.
\end{proof}

\section{Cardinality of $\PS_\varphi(\alpha)$}\label{sec-card-PS}

In this section, we characterize when $\PS_\varphi(\alpha)$ is finite and when it is a singleton.

\begin{theorem}\label{thm-ps-countable}
	For any~$\alpha$ in $ H^1(M,\ZZ)\setminus\{0\}$, the set $\PS_\varphi(\alpha)$ is at most countable.
\end{theorem}

\begin{proof}
	Denote by $C\subset[0,1]$ the standard Cantor (compact) set.
	Let $K\subset\wh M_\alpha$ be a compact subset for which the projection $K\to M$ is surjective. Then $K\cap\wh\Rec_\alpha$ is compact, so there exists a continuous surjection $\wh f\colon C\to K\cap\wh\Rec_\alpha$. Given any map $g$ in $\wh I_\alpha(\wh\Rec_\alpha^*,\ZZ)$, its value is completely determined by the images $g(R)$ for the set of the recurrent chains $R$ that intersects~$K$. Therefore, the continuous map $g\circ f\colon C\to\ZZ$ determines completely $g$. Thus, it is enough to prove that there are only countably many continuous maps from $C$ to $\ZZ$.

	Let $h\colon C\to\ZZ$ be a continuous map. Note that it is locally constant, so we can find a partition of $C$ by finitely many open sets $U_1\cdots U_n$, so that $h$ is constant of each $U_k$. Up to splitting each $U_k$ in finitely many disjoint open subsets, we may assume that $U_k$ can be written as $U_k=C\cap I_k$, where $I_k$ is an interval in $\RR$ whose ends are rational and outside $C$. 
	
	Note that there are only countably many intervals $I$ with $\partial I\subset\QQ$.
	Thus, $h$ is determined by the countable choices of a finite partition of $C$ by sets of the form $C\cap I_k$, and by the countable possible images of $h$ on $C\cap U_k$. Hence, there are no more than countably many continuous functions from $C$ to $\ZZ$.
\end{proof}

Denote by $\wh G_{\varphi,\alpha}$ the oriented graph whose set of vertices is $\wh\Rec_\alpha^*$ and with an edge from $R_1$ to $R_2$ if $R_1\recto R_2$ holds. We also denote by $\recto$ the edges inside $\wh G_{\varphi,\alpha}$. 
The graph $\wh G_{\varphi,\alpha}$ admits a natural $\ZZ$-action. 

\begin{definition}
	We denote by $G_{\varphi,\alpha}$ the oriented graph obtained as the quotient of $\wh G_{\varphi,\alpha}$ by $\ZZ$. Its vertices are the~$\alpha$-recurrence chains of $\varphi$. 
\end{definition}

The graph $G_{\varphi,\alpha}$ is \emph{transitive} if there exist an oriented path between any two~$\alpha$-recurrence chains. Notice that, any path $R_1\recto\cdots\recto R_n$ can be concatenated in a single arrow $R_1\recto R_n$, so $G_{\varphi,\alpha}$ is transitive if and only if it is a complete oriented graph. By convenient, we take the convention that the empty graph is transitive.

\begin{theorem}\label{thm-finite-pa}
	Let~$\alpha$ be a non-zero class in $H^1(M,\ZZ)$ so that $-\alpha$ is quasi-Lyapunov. Then $\PS_\varphi(\alpha)$ is finite if and only if the oriented graph $G_{\varphi,\alpha}$ is finite and transitive.
	
	The set $\PS_\varphi(0)$ is either empty or (infinitely) countable.
\end{theorem}

We warn the reader of the following. The two properties of ``$G_{\varphi,\alpha}$ being transitive'' and ``$\varphi$ being chain recurrent'' seem to be very similar. However, none of them implies the other one. Let us describe two counter-examples on the torus. 

\begin{example}\label{ex-non-trans}
	Take $M=\bfrac{\RR^2}{\ZZ^2}$ and let $\varphi$ be the following flow on $M$. It admits two periodic orbits $\delta_1=\{0\}\times\bfrac{\RR}{\ZZ}$ and $\delta_2=\{\tfrac{1}{2}\}\times\bfrac{\RR}{\ZZ}$. Every other orbit converges toward $\delta_0$ in the past and $\delta_1$ in the future. Also all its orbits are transverse to the curve $\gamma=\bfrac{\RR}{\ZZ}\times\{0\}$. Then, $\gamma$ is a global cross-section, cohomologous to $\alpha=[dy]$. Thus, $\varphi$ admits a unique partial cross-section cohomologous to $[dy]$, and $G_{\varphi,\alpha}$ is finite and transitive. But $\varphi$ is not recurrent, as its recurrent set is $\delta_1\cup\delta_2\neq M$.
\end{example}

\begin{example}\label{ex-non-trans-bis}
	Take $M=\bfrac{\RR^2}{\ZZ^2}$ and let $\varphi$ be the flow on $M$ obtained by taking the linear flow directed by $\partial y$, and slowing it down so that it admits exactly two fixed points $p=(0,0)$ and $q=(\tfrac{1}{2},0)$. We let the reader check that $\varphi$ is chain recurrent. We set $\alpha=[dy]$. The graph $G_{\varphi,\alpha}$ has two vertices, that correspond to $p$ and $q$, and its only edges are loops from a vertex to itself. Thus, it is not transitive, but $\varphi$ is chain recurrent.
\end{example}

\begin{proof}[Proof of Theorem~\ref{thm-finite-pa}, in the case $\alpha=0$]
	When $\varphi$ is chain recurrent, it follows from Theorem~\ref{thm-ps-to-non-rec} that $\PS_\varphi(0)$ is empty.
	
	We now assume that $\varphi$ is not chain recurrent. Then $\PS_\varphi(0)$ is in bijection with $I_\alpha(\wh\Rec_\alpha^*,\ZZ)$. For any $g$ in $I_\alpha(\wh\Rec_\alpha^*,\ZZ)$, it admits a unique representative $\wh g$ in $\wh I_\alpha(\wh\Rec_\alpha^*,\ZZ)$, which is not constant, and satisfies $\min(g)=0$. Then for any integer $k\geq 1$, $k\wh g$ also belongs to $\wh I_\alpha(\wh\Rec_\alpha^*,\ZZ)$. And for any distinct $i,j\geq 1$, $i\wh g$ and $j\wh g$ are not equivalent up to a constant. Thus, $I_\alpha(\wh\Rec_\alpha^*,\ZZ)$ is infinite.
\end{proof}


Before proving the theorem in the case $\alpha\neq0$, we build some explicit element in $\PS_\varphi(\alpha)$. Choose $\alpha\neq 0$ so that $-\alpha$ is quasi-Lyapunov. The case $\alpha(D_\varphi)>0$ is a consequence of Fried's work Theorem~\ref{thm-Fried}. So we may assume $\alpha(D_\varphi)\ngtr0$ here, and so $\Rec_\alpha\neq\emptyset$ holds according to Proposition~\ref{prop-part1-empty-rec}.

	By Theorem~\ref{thm-spectral-decomp}, there exists an~$\alpha$-equivariant function $\wh f\colon \wh M_\alpha\to\RR$ so that $-\wh f$ is Lyapunov. Denote by $f\colon M\to\bfrac{\RR}{\ZZ}$ the quotient map. Recall that $f(\Rec_\alpha)$ is non-empty, compact its interior is empty (see Lemma~\ref{lem-tot-disc}). So it admits a complementary component, which is an open interval in $\bfrac{\RR}{\ZZ}$. 
	Take $I\subset\bfrac{\RR}{\ZZ}$ a complementary component of~$f(\Rec_\alpha)$. We define $\pi_I\colon\bfrac{\RR}{\ZZ}\to\bfrac{\RR}{\ZZ}$ the map that consists in retracting the $\bfrac{\RR}{\ZZ}\setminus I$ in one point. That is $\pi_I$ is affine on $I$, constant outside $I$, and is of degree~1. We additionally assume that the complementary of $I$ is sent onto the point~$0$.
	Then $-\pi_I\circ f$ is pre-Lyapunov, cohomologous to $f$ and sends~$\Rec_\alpha$ into~$\{0\}$. That is $-\pi_I\circ f$ belongs to $[(M,\Rec_\alpha)\col(\bfrac{\RR}{\ZZ},\{0\})]_\alpha$. 
	
\begin{lemma}\label{lem-tech-preL-I}
	Let $I,J$ be two distinct complementary components of $f(\Rec_\alpha)$. Then $-\pi_I\circ f$ and $-\pi_J\circ f$ are not homotopic inside $[(M,\Rec_\alpha)\col(\bfrac{\RR}{\ZZ},\{0\})]_\alpha$.
\end{lemma}

\begin{proof}
	By definition, $I$ is bounded by two points in $f(\Rec_\alpha)$.
	Denote by $R_1,R_2$ two~$\alpha$-recurrence chains so that $f(R_1)$ and $f(R_2)$ bound $I$. Let $\gamma\subset M$ be a curve that starts on $R_1$ and end on $R_2$. We consider the image curve $f\circ\gamma$ in $\bfrac{\RR}{\ZZ}$. Note that $f\circ\gamma$ is homologous, relative to $\partial I$, to $I$ plus finitely many copies of $\bfrac{\RR}{\ZZ}$.	
	
	Take a third complementary interval~$K$ of $f(\Rec_\alpha)$. The curve $-\pi_K\circ f\circ\gamma$ in $\bfrac{\RR}{\ZZ}$ is a closed curve, since the two ends of $\gamma$ are sent onto $0$. Notice that the degree of $-\pi_K\circ f\circ\gamma$ is equal to the algebraic intersection between $-f\circ\gamma$ with any point inside~$K$. It follows from above that when taking $K=I$ or $K=J$, these algebraic intersections differ by $\pm1$.
	Thus, the degrees of $-\pi_I\circ f\circ\gamma$ and $-\pi_J\circ f\circ\gamma$ are different.
	
	For maps $g$ in $[(M,\Rec_\alpha)\col(\bfrac{\RR}{\ZZ},\{0\})]_\alpha$, the degree of $g\circ\gamma$ is an invariant of homotopy inside that set. 
	Thus, $-\pi_I\circ f$ and $-\pi_J\circ f$ are not homotopic in that set. 
\end{proof}

Denote by $\Rec_\alpha^*$ the set of~$\alpha$-recurrence chains, equipped with the quotient topology using the projection $\Rec_\alpha\to\Rec_\alpha^*$.

\begin{proof}[Proof of Theorem~\ref{thm-finite-pa}, in the case $\alpha\neq0$]
	Take a non-zero class~$\alpha$ so that $-\alpha$ is quasi-Lyapunov.
	When~$\Rec_\alpha$ is empty, $G_{\varphi,\alpha}$ is finite and transitive by convention. And $\PS_\varphi(\alpha)$ is a singleton from Proposition~\ref{prop-part1-empty-rec} and Theorem~\ref{thm-Fried}. So the conclusion holds in that case.

	We now assume that~$\Rec_\alpha$ is not empty.
	The set $\PS_\varphi(\alpha)$ is in bijection with $I_\alpha(\wh\Rec_\alpha^*,\ZZ)$, so it is enough to determine the cardinality of the later set.
	Let $n_\alpha\geq 1$ be so that $\alpha=n_\alpha\beta$ for some primitive class~$\beta$ in $H^1(M,\ZZ)$.

	We suppose that $G_{\varphi,\alpha}$ is finite and transitive, and prove that $\PS_\varphi(\alpha)$ is finite. Denote by $R_1\cdots R_n$ the~$\alpha$-recurrence chains on $M$, and a lift $\wt R_i$ of~$R_i$ in~$\wh M_\alpha$.	
	By transitivity of $G_{\varphi,\alpha}$, for every $i\neq j$, there exists $a_{i,j}$ in $\ZZ$ that satisfies $\wt R_i\recto a_{i,j}\cdot\wt R_j$. 

	Take an element in $I_\alpha(\wh\Rec_\alpha^*,\ZZ)$, and its unique lift $g$ in $\wh I_\alpha(\wh\Rec_\alpha^*,\ZZ)$ which satisfies $g(R_1)=0$. 
	From above, we have $g(R_i)\geq g(R_j) + n_\alpha a_{i,j}$ for any $i\neq j$.
	So $g(R_i)$ lies in $\intint{-n_\alpha a_{1,i},n_\alpha a_{i,1}}$. Thus, $g$ is determined by a finite data. It follows that $I_\alpha(\wh\Rec_\alpha^*,\ZZ)$ is finite.

	We now prove the converse implication by contraposition. 
	Fix an~$\alpha$-equivariant map $\wh f\colon \wh M_\alpha\to\RR$ so that $-\wh f$ is Lyapunov (see Theorem~\ref{thm-spectral-decomp}), which project down to a map $f\colon M\to\bfrac{\RR}{\ZZ}$. 
	Assume first that $G_{\varphi,\alpha}$ is infinite. Then $\Rec_\alpha^*\simeq f(\Rec_\alpha)$ is infinite and totally disconnected. So there are infinitely many complementary intervals of $f(\Rec_\alpha)$ in $\bfrac{\RR}{\ZZ}$. 
	It follows from Lemma~\ref{lem-tech-preL-I} that $I_\alpha(\wh\Rec_\alpha^*,\ZZ)$ is infinite.

	Assume now that the quotient graph $G_{\varphi,\alpha}$ is not transitive. We may also choose it finite for convenience. Since the graph is not transitive, there exists a partition $\{A,B\}$ of the set of vertices of $G_{\varphi,\alpha}$ so that there is no arrow from $B$ to $A$. Denote by $\wh A$ and $\wh B$ the preimages of $A$ and $B$ inside $\wh G_{\varphi,\alpha}$. 
	
	Since $\PS_\varphi(\alpha)$ is not empty (see Theorem~\ref{thm-ps-to-non-rec}), there exists a map $h$ in $\wh I_\alpha(\wh\Rec_\alpha^*,\ZZ)$. For all $k$ in $\NN$, we define the map $h_k\colon \wh\Rec_\alpha^*\to \ZZ$ that coincide with $h$ on $\wh A$ and with $h+k$ on~$\wh B$. It is clear that $h_k$ is non-creasing and continuous (by finiteness of $\Rec_\alpha^*$), so it belongs to $\wh I_\alpha(\wh\Rec_\alpha^*,\ZZ)$. Observes that for any distinct indexes $i,j$, the functions $h_i$ and $h_j$ do not differ by a constant, so their images in $I_\alpha(\wh\Rec_\alpha^*,\ZZ)$ are distinct. It follows that $I_\alpha(\wh\Rec_\alpha^*,\ZZ)$ is infinite.
\end{proof}

Using the same technique, we can determine when there is a unique partial cross-section up to isotopy.

\begin{theorem}\label{thm-unique-ps}
	Let~$\alpha$ be in $H^1(M,\ZZ)$. The set $\PS_\varphi(\alpha)$ is a singleton if and only if~$\alpha$ is non-zero, $-\alpha$ is quasi-Lyapunov and if there exists at most one~$\alpha$-recurrence chain.
\end{theorem}

\begin{proof}
	Assume $\alpha\neq 0$ and quasi-Lyapunov, otherwise it follows from above that $\PS_\varphi(\alpha)$ is not a singleton.
	Let $f\colon M\to\bfrac{\RR}{\ZZ}$ be a continuous map, cohomology to~$\alpha$, and $\wh f\colon\wh M_\alpha\to\RR$ a lift of $f$, so that $-\wh f$ is Lyapunov.

	Assume that there exists at least two~$\alpha$-recurrence chains. Then the space $f(\Rec_\alpha)\simeq\Rec_\alpha^*$ has at least two complementary components. It follows from Lemma~\ref{lem-tech-preL-I} that $\PS_\varphi(\alpha)$ has at least two elements. Thus, if $\PS_\varphi(\alpha)$ is a singleton, $\varphi$ has at most one~$\alpha$-recurrence chain. 

	Conversely, if there is at most one~$\alpha$-recurrence chain, then $\wh I_\alpha(\wh\Rec_\alpha^*,\ZZ)$ is either empty (when~$\Rec_\alpha$ is empty) or equal to a single $\ZZ$-orbit. Thus, $I_\alpha(\wh\Rec_\alpha^*,\ZZ)$ is a singleton, and so is $\PS_\varphi(\alpha)$. 
\end{proof}

\section{Germs in homology}\label{sec-germ-hom}

We will discuss a homology construction. It is certainly well-known by some expert, but we could not find it in the bibliography. And it is not a bad idea to reintroduce it here anyway. In this section, we develop a general theory independent on the flow. We come back to the subject of partial cross-sections in Section~\ref{sec-hom-classification}.

Given a compact subset~$K$ in $M$, we build a homology module relative to the germ of neighborhoods of~$K$, as a projective limit of some singular homology modules. It can be though of the singular homology relative to a compact subset~$K$, modulo arbitrarily small neighborhoods of~$K$.

\subsection{Homology relative to a germ}


Let us fix some notations. We take $\A$ to be one of the ring $\ZZ$ or $\RR$. 
We consider the singular homology and cohomology with coefficients in $\A$. 
Given two sets $Y\subset X$, we sometimes write $H_1(X)$ and $H_1(X,Y)$ the modules $H_1(X,\A)$ and $H_1(X,Y,\A)$, when the chosen ring is clear from the context.
Take four sets $A,B,X,Y$ that satisfy $Y\subset X$, $A\subset B$, $A\subset X$ and $B\subset Y$. The inclusion $(A,B)\hookrightarrow(X,Y)$ induces a morphism $H_1(A,B)\to H_1(X,Y)$. Given an element $\delta$ in $H_1(A,B)$, we will denote by \emph{$\delta\hmod{X,Y}$} its image in $H_1(X,Y)$. 
Or simply write $\delta\rmod{X}$ when $B=Y$ holds, and $\delta\lmod{Y}$ when $A=X$ holds.
We use a similar notation for the map $H^1(X,Y)\to H^1(A,B)$.

We will discuss a projective limit of some modules. For more on the matter, we refer to \cite[\href{https://stacks.math.columbia.edu/tag/0031}{Tag 0031}]{Stacks} for directed inverse systems, and to \cite[\href{https://stacks.math.columbia.edu/tag/002U}{Tag 002U}]{Stacks} for projective limits. 

Fix a compact subset~$K$ of $M$, potentially empty.
The set of neighborhoods of $K$ is a \emph{directed inverse system} for the partial order~$\supset$. That is, for any two neighborhood $U,V$, there exists a neighborhood $W$, for instance $W=U\cap V$, that satisfies $U\supset W$ and $V\supset W$. It is a necessary condition to define a projective limit.

Consider the following diagram of $\A$-modules. The objects are the modules $H_1(M,U,\A)$ where~$U$ varies over the neighborhoods of $K$. The diagram contains the maps $H_1(M,V,\A)\xrightarrow{\srmod} H_1(M,U,\A)$, where $U\supset V$ are any two neighborhoods of $K$. We use the notation $\germ(K)$ for the germ of neighborhoods of~$K$.

\begin{definition}
	We denote by:
	$$H_1(M,\germ(K),\A)=\varprojlim_{U\supset K}H_1(M,U,\A)$$
	the projective limit, equipped with the maps 
	$$H_1(M,\germ(K),\A)\xrightarrow{\srmod}H_1(M,U,\A)$$ 
	for any neighborhood~$U$ of~$K$.
\end{definition}

By definition, the projective limit is equal to:
$$\varprojlim_{U\supset K}H_1(M,U,\A)=\left\{\delta\in\prod_{U\supset K}H_1(M,U,\A),\,\forall U\subset V,\,\delta_U\rmod{V}=\delta_V\right\}$$

Given a subset $X$ of $M$, we denote by $V_\epsilon(X)$ the set of points at distance less than $\epsilon$ of $X$.
Since~$K$ is compact, any neighborhood of~$K$ some $V_\epsilon(K)$ with a small $\epsilon>0$. So for convenience, one may replace~$U$ by $V_\epsilon(K)$ in the projective limit.

The projective limit satisfies a universal property. So the inclusion maps $H_1(M,K)\to H_1(M,U)$ induce a natural map $H_1(M,K)\to H_1(M,\germ(K))$. We will discuss more about this map further down.

An open set $U\supset K$ is said to be \emph{$K$-connected} if any path-connected component of~$U$ intersects~$K$, or alternatively $H_0(U,K)=\{0\}$ holds. 
Note that $V_\epsilon(K)$ is $K$-connected.

\begin{lemma}\label{lem-germ-surf}
	The map $H_1(M,\germ(K))\xrightarrow{\srmod} H_1(M,U)$ is surjective for any $K$-connected neighborhood~$U$ of~$K$.
\end{lemma}

\begin{proof}
	The long exact sequence in homology contains the maps
	$$H_1(M,K)\to H_1(M,U)\to H_0(U,K)=0.$$
	So the left map is surjective. Since it can be factorized as
	$$H_1(M,K)\to H_1(M,\germ(K))\to H_1(M,U),$$
	the second map is surjective.
\end{proof}

\begin{remark}
	The map $H_1(M,K)\to H_1(M,\germ(K))$ is neither injective nor surjective in general, as illustrated in Examples~\ref{ex-non-surj} and~\ref{ex-non-inj}. We will characterize its image in Proposition~\ref{prop-img-hom-finit-l}.
\end{remark}

Let us denote temporally by $\pi_\germ\colon H_1(M,K)\to H_1(M,\germ(K))$ this map.

\begin{example}\label{ex-non-surj}
	Consider $M=[0,1]$ and $K=\{0\}\cup\{\frac{1}{n+1},n\in\NN\}$. The formal infinite 1-chain:
	$$\delta_1=\sum_{n\in\NN}\left[0,\lsfrac{1}{n+1}\right]$$
	induces an element in $H_1(M,\germ(K))$, since only finitely many terms appear modulo any given neighborhood of~$K$. Clearly, $\delta_1$ is not in the image of $\pi_\germ$. So the map $\pi_\germ$ is not surjective.
\end{example}

\begin{example}\label{ex-non-inj}
	Take $M=(\bfrac{\RR}{\ZZ})\times\RR$ and the compact subset
	$$K=\left\{\!\middle(t,\,\sin\!\middle(\lsfrac{1}{t}\middle)\!\middle)\!,\,t\in]0,1]\middle\}\cup\middle(\{0\}\times[-1,1]\right)$$
	of $M$. Consider the curve $\delta_2$ that goes from $[0,\frac{1}{2\pi}]\times\{0\}$. It is not null homologous in $H_1(M,K)$. But $\delta_2$ is null-homologous in $H_1(M,\germ(K))$, since for any neighborhood~$U$ of~$K$, $\delta_2$ is homotopic (relatively to $K$) to a curve that remains inside~$U$. So the map $\pi_\germ$ is not injective.
\end{example}

Replacing~$K$ with~$\Rec_\alpha$, the second example corresponds to a curve that has `zero contribution to the dynamic'. It is actually to remove these curves that we consider the module $H_1(M,\germ(K))$ in the first place.

\begin{remark}
	One can define similarly the homology modules $H_n(\germ(K),\A)$ and $H_n(\germ(K),K,\A)$. One difficulty is that the long exact sequence in homology are more complicated when using the germs. In particular the expected two sequences 
	$$\cdots\to H_n(K)\to H_n(\germ(K))\to H_n(M)\to H_{n-1}(K)\to\cdots$$
	and
	$$\cdots\to H_n(\germ(K),K)\to H_n(M,K)\to H_n(M,\germ(K))\to H_{n-1}(\germ(K),K)\to\cdots$$
	are not exact in general. The reason is that the directed inverse systems given by the modules $H_n(U)$ and $H_n(U,K)$ are not Mittag-Leffler. This goes beyond the scope of this paper, so we refer to \cite[\href{https://stacks.math.columbia.edu/tag/0595}{Tag 0595}]{Stacks} for the Mittag-Leffler criterion, and \cite[\href{https://stacks.math.columbia.edu/tag/0H31}{Tag 0H31}]{Stacks} for an exact sequences that one can hope to obtain.
\end{remark}

\subsection{Torsion and extension of coefficients}

We consider $\A=\RR$ and $\ZZ$ for different reasons. It is easier to work with $\RR$ since it is a field with a good topology. But the natural ring to express the classification of partial cross-sections is $\ZZ$. Let us explain how they relate to each other.

Given a $\ZZ$-module $H$, denote by $\tor(H)$ the torsion part of $H$, and $\bfrac{H}{\tor}$ the quotient of $H$ by $\tor(H)$.

\begin{lemma}\label{lem-tor-bounded}
	Let~$U$ be a neighborhood of~$K$. The two maps 
	$$\tor(H_1(M,U,\ZZ))\xrightarrow{\srmod}\tor(H_1(M,\germ(K),\ZZ))\xrightarrow{\srmod}\tor(H_1(M,U,\ZZ))$$
	are surjective. When~$U$ is a small enough neighborhood of $K$, the right map is an isomorphism.
\end{lemma}

\begin{proof}
	Here, all homology modules are with coefficients in $\ZZ$. The long exact sequence in homology for $(U,M)$ contains the maps:
	$$H_1(M)\xrightarrow{\srmod} H_1(M,U)\xrightarrow{\partial} H_0(U).$$
	Since $H_0(U)$ is a free Abelian group, the torsion part $\tor(H_1(M,U))$ lies inside $\ker(\partial)=\im(\srmod)$.
	So the map $\tor(H_1(M))\to\tor(H_1(M,U))$ is surjective. Note that the cardinality of the image is non-increasing in~$U$, and bounded by the cardinality of $\tor(H_1(M))$, which is finite. So the cardinality of the image is constant when~$U$ is small enough. 
	So given two neighborhoods $V\subset U$ of $K$, assumed small enough, the map $\tor(H_1(M,V))\to\tor(H_1(M,U))$ is an isomorphism.
	It follows that $\tor(H_1(M,\germ(K),\ZZ))\to\tor(H_1(M,U,\ZZ))$ is also an isomorphism.
\end{proof}

\begin{corollary}
	The quotient $\bfrac{H_1(M,\germ(K),\ZZ)}{\tor}$ is naturally isomorphic to the projective limit:
	$$\bfrac{H_1(M,\germ(K),\ZZ)}{\tor}\simeq\varprojlim_{U\supset K}\left(\bfrac{H_1(M,U,\ZZ)}{\tor}\right).$$
\end{corollary}

\begin{proof}
	From above, we consider the directed inverse system of exact sequences:
	$$1\mapsto\tor(H_1(M,U))\to H_1(M,U)\to\bfrac{H_1(M,U)}{\tor}\to 1,$$
	where~$U$ varies over the neighborhoods of $K$. By Lemma~\ref{lem-tor-bounded}, the first term is constant when~$U$ is small enough. Consider the projective limit of these sequences. Note that the projective limit functor is left exact, but not exact in general. But since the left term is eventually constant, it follows from 
	\cite[\href{https://stacks.math.columbia.edu/tag/0H31}{Tag 0H31}]{Stacks}
	and 
	\cite[\href{https://stacks.math.columbia.edu/tag/07KW}{Tag 07KW, item 3.}]{Stacks}
	that the short sequence 
	$$1\mapsto\tor(H_1(M,\germ(K)))\to H_1(M,\germ(K))\to\varprojlim_{U\supset K}\bfrac{H_1(M,U)}{\tor}\to 1$$
	is exact. The conclusion follows.
\end{proof}

\begin{remark}
	We later introduce the cohomology module $H^1(M,\germ(K),\ZZ)$, which is the continuous dual of $H_1(M,\germ(K),\ZZ)$. As such, it does not see the torsion part of $H_1(M,\germ(K),\ZZ)$. Thus, one could replace $H_1(M,\germ(K),\ZZ)$ by $\bfrac{H_1(M,\germ(K),\ZZ)}{\tor}$ in many of our results. In order to remain general, we will not make this choice. 
\end{remark}

\begin{lemma}\label{lem-coeff-extension}
	The coefficient extension maps $H_1(M,U,\ZZ)\to H_1(M,U,\RR)$ induces a map:
	$$H_1(M,\germ(K),\ZZ)\to H_1(M,\germ(K),\RR)$$
	which is $\ZZ$-linear, continuous, and its kernel is $\tor(H_1(M,\germ(K),\ZZ))$.
\end{lemma}


\begin{proof}
	Notice the composition of maps 
	$$H_1(M,\germ(K),\ZZ)\to H_1(M,U,\ZZ)\to H_1(M,U,\RR).$$
	Since $H_1(M,U,\RR)$ satisfies a universal property in the set of $\ZZ$-module, this maps induce a limiting map
	$$H_1(M,\germ(K),\ZZ)\to H_1(M,\germ(K),\RR).$$
	It follows from the definitions that it is $\ZZ$-linear and continuous. Take $x$ in its kernel. For any neighborhood~$U$ of~$K$, the image of $x$ in $H_1(M,U,\RR)$ is equal to zero, so the image of $x$ modulo~$U$ belongs to $\tor(H_1(M,U,\ZZ))$. Taking~$U$ small enough, it follows from Lemma~\ref{lem-tor-bounded} that $x$ belongs to the torsion part $\tor(H_1(M,\germ(K),\ZZ))$.
\end{proof}

\subsection{Simplicial norm}

In this section, we determine the image of $H_1(M,K,\A)$ in $H_1(M,\germ(K),\A)$, using the simplicial norm.
Given a subset $X\subset M$, we denote by $\sell_X$ the simplicial norm on $H_1(M,X)$, or simply $\sell$ when $X$ is clear from the context. That is, given some $\delta$ in $H_1(M,X)$, $\sell_X(\delta)$ the minimum $n$ for which we can represent $\delta$ as a sum:
$$\delta\equiv\sum_{i=1}^na_i\delta_i$$
of $n$ curves $\delta_i$, with $a_i$ in $\A$. 
Given two subsets $X\subset Y$ of~$M$, notice that the simplicial norm satisfy $\sell_Y(\delta\rmod{Y})\leq\sell_X(\delta)$ for any $\delta$ in~$H_1(M,X)$. We define the function $\sell=\sell_{\germ(K)}\colon H_1(M,\germ(K))\to[0,+\infty]$, called the \emph{simplicial norm}, by:
$$\sell_{\germ(K)}(\delta)=\sup_{U\supset K}\sell_U(\delta\rmod{U})$$
where~$U$ varies over the neighborhoods of~$K$. 

Note that the curve $\delta_1$ in Example~\ref{ex-non-surj} satisfies $\sell_{\germ(K)}(\delta_1)=+\infty$. So it is necessary to include the value $+\infty$.
For any $\delta$ in $H_1(M,K)$, similarly to above, we have $\sell_{\germ(K)}(\delta\rmod{\germ(K)})\leq\sell_K(\delta)$. Note that the equality is strict for the element $\delta_2$ in Example~\ref{ex-non-inj}, Since $\delta_2\rmod{\germ(K)}=0$ holds.

\begin{proposition}\label{prop-img-hom-finit-l}
	Recall $\A=\RR$ or $\ZZ$.
	The image of $H_1(M,K,\A)$ modulo $\germ(K)$ is the set
	$$H_1(M,K,\A)\rmod{\germ(K)}=(\sell_{\germ(K)})^{-1}([0,+\infty[).$$
	of elements that have finite simplicial norms.
	Additionally, for any class $\delta$ in $H_1(M,\germ(K),\A)$ with $\sell_{\germ(K)}(\delta)<+\infty$, there exists $\gamma$ in $H_1(M,K,\A)$ that satisfies $\sell_K(\gamma)=\sell_{\germ(K)}(\delta)$ and $\gamma\rmod{\germ(K)}=\delta$.
\end{proposition}

The conclusion is trivial when~$K$ is empty, as $H_1(M,K)\to H_1(M,\germ(K))$ is an isomorphism in that case. So let us assume~$K$ non-empty from now on. 

The proof is technical, so let use summarize it first. We take a decreasing sequence $U_n$ of open neighborhoods of~$K$. Take $\delta$ in $H_1(M,\germ(K))$ that has a finite simplicial norm, and search a preimage in $H_1(M,K)$. One may want to consider $\delta\rmod{U_n}$, which is a brave homology class representable by a 1-chain $\gamma_n$ with $\partial\gamma_n\in\partial U_n$. Then one may want to define the element $\gamma_{n+1}$ by extending $\gamma_n$ inside $U_n\setminus U_{n+1}$. The finite simplicial norm ensure that it can be done with a bounded number of curves. 

One difficulty is that the limit object can be a combination of curves that have infinite lengths. 
A second difficulty is that we need to control the coefficients in $\gamma_n$. So instead, we first prove that, up to an extraction, the boundary of $\gamma_n$ converges. Then we kind of apply the idea above.

\begin{lemma}\label{lem-pre-image-mod-U}
	Let $\delta$ in $H_1(M,\germ(K),\A)$ satisfy $\sell(\delta)<+\infty$. Then there exists $p$ in $H_0(K,\A)$ which satisfies the following. For any neighborhood~$U$ of~$K$, there exists $\gamma$ in $H_1(M,K,\ZZ)$ that satisfies $\sell(\gamma)\leq\sell(\delta)$, $\partial\gamma=p$, and with $\gamma\rmod{U}=\delta\rmod{U}$.
\end{lemma}

\begin{proof}
	We choose a decreasing sequence $U_n$ of neighborhoods of~$K$ that satisfy the following. First we assume $\sell_{U_n}(\delta\rmod{U_n})=\sell(\delta)$, which holds true whenever $U_n$ is small enough. Secondly, we assume that $U_n$ is contained in $V_{\frac{1}{n}}(K)$, so that the sequence of values of $\delta\rmod{U_n}$ determines the value of $\delta$. 
	Denote by $\gamma_n$ a 1-chain that represent $\delta\rmod{U_n}$, and that can be written as 
	$$\gamma_n=\sum_{i=1}^{\sell(\delta)}a_{n,i}c_{n,i}$$
	with $a_{n,i}$ in $\A$, and where $c_{n,i}$ is a curve with boundary in $U_n$. Up to an extraction, we may assume that the points in $\partial c_{n,i}$ accumulates on finitely many connected components $R_1\cdots R_p$ of~$K$. 
	
	Since $K$ is compact, any connected component of $K$ is also compact.
	Consider the set of connected components of $K$, with the quotient topology. It is metrizable (using the Hausdorff distance) and totally disconnected. So each pairs of distinct connected components of $K$ are separable by open sets. Thus, when $\epsilon>0$ is small enough, each two distinct $R_i$ lie in distinct connected components of $V_\epsilon(K)$.

	By removing finitely many $U_n$, we may assume that the $R_j$ lie in distinct connected components of $U_n$. Up to replacing $\gamma_n$ by a $\gamma_{k_n}$ for some large $k_n\geq n$, we may assume that the boundary points of $\gamma_n$ lie in the connected components of $U_n$ that contain the $R_j$.

	Pick a point $x_j$ in each~$R_j$.
	Then, for any endpoints $q$ of $\gamma_n$, there exists a unique $x_j$ in the same connected component of $U_n$ than $q$. Define an element $p_n$ in $H_0(K)$ by replacing any point $q$ in $\partial\gamma_n$ with the corresponding $x_j$. Then $p_n$ is supported by the point $x_j$, and it satisfies:
	$$p_n\lmod{U_n}=\partial\gamma_n\lmod{U_n}=\partial(\delta\rmod{U_n}).$$

	For any $m\geq n$, we have $p_n-p_m\lmod{U_n}=0$. The difference $p_n-p_m$ is supported on the $x_j$, and each two $x_j$ lie in two distinct connected components of $U_n$. So we actually have $p_n=p_m$. Denote by $p$ the common value of the $p_n$. By construction, we may extend the curves in $\gamma_n$ inside $U_n$, so that that its boundary lies on the $x_j$, and so that $\partial\gamma_n=p$ holds. In particular the boundary of $\gamma_n$ lies in $K$, so $\gamma_n$ induces a homology class in $H_1(M,K,\ZZ)$.

	For any neighborhood $V$ of $K$, there exists a neighborhood $U_n$ included in~$V$. Then the homology class of $\gamma_n$ in $H_1(M,K,\ZZ)$ satisfies the conclusion of the lemma.
\end{proof}

A neighborhood~$U$ of~$K$ is said \emph{$K$-thin} if for any other neighborhood $V\subset U$ of~$K$, we have:
$$H_1(U,\RR)\lmod{M}=H_1(V,\RR)\lmod{M}$$
as subsets of $H_1(M,\RR)$.
Since $H_1(U,\RR)\lmod{M}$ is non-decreasing in~$U$, and it lies in a finite dimensional vector space, it stabilizes when~$U$ is small enough. Note that $K$-thin is a property of the homology with real coefficients only.

\begin{proof}[Proof of Proposition~\ref{prop-img-hom-finit-l}]
	We may assume $\delta\neq 0$ since this the case $\delta=0$ is trivial.
	Choose a sequence $U_n$ of $K$-thin neighborhoods of $K$, included in $V_{\frac{1}{n}}(K)$.	
	From Lemma~\ref{lem-pre-image-mod-U}, there exists $p$ in $H_0(K)$ and for any $n$, an element $\gamma_n$ in $H_1(M,K,\A)$ that satisfies $\partial\gamma_n=p$, $\sell(\gamma)\leq\sell(\delta)$ and $\gamma_n\rmod{U_n}=\delta\rmod{U_n}$.

	We first consider the case $\A=\RR$.
	We claim that for any $m\geq n$, we have $(\gamma_n-\gamma_m)\rmod{U_m}=0$.
	The claim implies $\gamma_0\rmod{U_n}=\delta\rmod{U_n}$ for all $n$, and then
	$\gamma_0\rmod{\germ(K)}=\delta$ by construction of the projective limit. We have $\sell(\gamma_0)\leq\sell(\delta)$ by definition, and the equality holds since the map $\prmod{\germ(K)}$ does not increase the simplicial norm.
	It implies the conclusion of the lemma in the case~$\A=\RR$.
	
	Let us prove the claim. 
	We consider the following commutative diagram, where the two horizontal lines are exacts:
	\definecolor{lightergray}{gray}{0.6} 
	\[
	\begin{tikzcd}[column sep=0.5cm]
		 & H_1(M,\RR) \arrow[r] & H_1(M,K,\RR) \arrow[r,"\partial"] & H_0(K,\RR) \\[-1cm]
		H_1(U_m,\RR) \arrow[ur] & |[lightergray]| y \arrow[r, lightergray, mapsto] & |[lightergray]| \gamma_n-\gamma_m \arrow[r, lightergray, mapsto] \arrow[dddd] \arrow[ddddr] & |[lightergray]| 0 \\[-0.8cm]
	|[lightergray]| w \arrow[ur, lightergray, mapsto]\\[-0.7cm]
		& H_1(U_n,\RR) \arrow[uu] \\[-0.9cm]
		& |[lightergray]| z \arrow[d] \\[-0.2cm]
		& H_1(U_n,U_m,\RR) \arrow[r] & H_1(M,U_m,\RR) \arrow[r] & H_1(M,U_n,\RR) \\[-1cm]
		& |[lightergray]| x \arrow[r, lightergray, mapsto] & |[lightergray]| (\gamma_m-\gamma_n)\rmod{U_m} \arrow[r, lightergray, mapsto] & |[lightergray]| 0 \\
	\end{tikzcd}
	\]

	Fix some $m>n$. By definition of $\gamma_n$, we have $(\gamma_m-\gamma_n)\rmod{U_n}=0$.
	The lower exact sequence in the diagram implies that there exists some $x$ in $H_1(U_n,U_m,\RR)$ that satisfies $(\gamma_m-\gamma_n)\rmod{U_m}=x\lmod{M}$.
	We have $\partial\gamma_n=p=\partial\gamma_m$, so $\partial(\gamma_n-\gamma_m)=0$ holds. Thus, the upper sequence implies that there exists $y$ in $H_1(M,\RR)$, that satisfies $\gamma_n-\gamma_m=y\rmod{K}$.
	The relative Mayer-Vietoris sequence applied to the sequence:
	$$(U_n,\emptyset)\hookrightarrow(U_n,U_m)\oplus(M,\emptyset)\hookrightarrow(M,U_m)$$
	yields that the sequence:
	$$\cdots\to H_1(U_n,\RR)\xrightarrow{p\mapsto(p,-p)}H_1(U_n,U_m,\RR)\oplus H_1(M,\RR)\to H_1(M,U_m,\RR)\to\cdots$$ 
	is exact. Note that $(x,-y)$ lies in the kernel of the second map. So there exists $z$ in $H_1(U_n,\RR)$ which $z\lmod{M}=y$ and $z\rmod{U_m}=x$. 
	Since $U_n$ is~$K$-thin, and since we consider real coefficients, there exists $w$ in $H_1(U_m,\RR)$ with $w\lmod{M}=z\lmod{M}=y$. The diagram commutes, so we have:
	$$(\gamma_m-\gamma_n)\rmod{U_m}=y\rmod{U_m}=w\hmod{M,U_m}=0$$
	since $w$ is supported by $U_m$. The claim follows.

	We now consider the case $\A=\ZZ$. We first work modulo the torsion. Denote by $\delta_t$ the image of $\delta$ in $\bfrac{H_1(M,\germ(K),\ZZ)}{\tor}$. Then the images of $\gamma_n$ and $\delta_t$ in $\bfrac{H_1(M,U_n,\ZZ)}{\tor}$ are equal. 
	Using the first case, we know that for any $n\geq 0$, the image of $\gamma_0-\gamma_n$ in $H_1(M,U_n,\RR)$ is zero. So by Lemma~\ref{lem-coeff-extension}, the image of $\gamma_0-\gamma_n$ in $\bfrac{H_1(M,U_n,\ZZ)}{\tor}$ is also zero. It follows that the image of $\gamma_0$ in $\bfrac{H_1(M,\germ(K),\ZZ)}{\tor}$ is equal to $\delta_t$. So $\gamma_0\rmod{\germ(K)}-\delta$ lies in $\tor(H_1(M,\germ(K),\ZZ))$. From Lemma~\ref{lem-tor-bounded}, there exists $t$ in $\tor(H_1(M,\ZZ))$ with $(\gamma_0+t)\rmod{\germ(K)}=\delta$. It follows that $\delta$ is the image of $\gamma_0+t\rmod{K}$.

	We show that $\sell(\gamma_0+t\rmod{K})=\sell(\gamma_0)\leq\sell(\delta)$ holds, which implies $\sell(\gamma_0)=\sell(\delta)$ as above.
	Note that $\gamma_0$ is not zero since $\delta$ is taken non-zero. 
	So take a representation $c_0$ of $\gamma_0$ as a combination of $\sell(\gamma_0)>0$. 
	Since $M$ is connected and since $t$ is closed, $t$ can be represented using only one curve~$c_1$.
	Take a point $p_0$ in $c_0$, a point $p_1$ in $c_1$, a curve $d$ from $p_0$ to $p_1$, and $\wb d$ the curve $d$ with the opposite orientation. We cut $c_0$ at $p_0$, and concatenate to it the curve $\wb d\circ c_1\circ d$. It yields a 1-chain with $\sell(\gamma_0)$ curves, and homologous to $\gamma_0+t$. It follows $\sell(\gamma_0+t\rmod{K})=\sell(\gamma_0)$, which ends the case $\A=\ZZ$.
\end{proof}

\subsection{The projective limit topology}

We equip $H_1(M,\germ(K))$ with a topology, which is non-trivial even for~$\A=\ZZ$.

\begin{lemma}\label{lem-K-con-finit-gen}
	Recall $\A=\ZZ$ or $\RR$. For any $K$-connected neighborhood~$U$ of~$K$, the modules $H_0(U,\A)$ and $H_1(M,U,\A)$ are finitely generated.
\end{lemma}

\begin{proof}
	Since~$K$ is compact, the set of path-connected components of~$U$ is finite. So $H_0(U)$ is finitely generated.
	The long exact sequence in homology, for $U\hookrightarrow M$, contains the maps:
	$$H_1(M)\to H_1(M,U)\to H_0(U).$$
	The left and right terms are finitely generated, and $\A$ is Noetherian. It implies that middle term is also finitely generated. 
\end{proof}

As a consequence, when~$U$ is $K$-connected, $H_1(M,U)$ is equipped with a canonical topology, defined by any norm. When $\A=\ZZ$ holds, the topology on $H_1(M,U)$ is trivial, but it is still useful to define a topology on $H_1(M,\germ(K),\ZZ)$.

\begin{definition}
	We equip $H_1(M,\germ(K),\A)$ with the coarsest topology that makes all the maps $H_1(M,\germ(K),\A)\xrightarrow{\srmod}H_1(M,U,\A)$ continuous, for the $K$-connected neighborhoods~$U$ of~$K$. It is called the \emph{projective limit topology}.
\end{definition}

Given a $K$-connected neighborhood~$U$ of $K$ and $\mathcal{V}\subset H_1(M,U,\A)$ open, the pre-image of $\mathcal{V}$ in $H_1(M,\germ(K),\A)$ is open. 
For $\A=\ZZ$, we can even take $\mathcal{V}$ to be a singleton.
The collection of these sets form a basis of the projective limit topology. 
It follows immediately:

\begin{lemma}
	The set $H_1(M,K,\A)\rmod{\germ(K)}$ is dense in $H_1(M,\germ(K),\A)$.
\end{lemma}

\begin{proof}
	Take $\delta$ in $H_1(M,\germ(K))$, and an open neighborhood $\mathcal{V}$ of $\delta$. Then there exists an open neighborhood~$U$ of $K$ so and a neighborhood $\mathcal{V'}$ of $\delta\rmod{U}$ in $H_1(M,U)$, so that the preimage of $\mathcal{V'}$ in $H_1(M,\germ(K))$ contains $\delta$. 
	
	Up to taking~$U$ smaller, we may assume it $K$-connected. Then the map $H_1(M,K)\to h_1(M,U)$ is surjective, so there exists $\gamma$ in $H_1(M,\germ(K))$ with $\gamma\rmod{U}=\delta\rmod{U}$. It follows that $\gamma\rmod{\germ(K)}$ lies in $\mathcal{V}$.
\end{proof}

We give below some elementary properties of the projective limit topology.
For any $n\geq 1$, fix a distance $d_n$ on $H_1(M,V_{\frac{1}{n}}(K),\A)$, and denote by $d_n^*$ the push back of $d_n$ on $H_1(M,\germ(K),\A)$. We denote by $d^*_\infty$ the distance on $H_1(M,\germ(K),\A)$ defined by 
$$d^*_\infty(\delta,\gamma)=\sum_{n\geq 1}\frac{1}{2^n}\frac{d^*_n(\delta,\gamma)}{1+d^*_n(\delta,\gamma)}.$$

We let the reader verify that it is indeed a distance, and that the following holds:

\begin{lemma}
	$d^*_\infty$ is a distance compatible with the projective limit topology. 
\end{lemma}

\begin{remark}
	Going back to Example~\ref{ex-non-inj}, the set $H_1(M,\germ(K),\A)$ has no compatible norm. Indeed, if it had a norm, there would be a sequence $a_n$ in $\NN$, that goes to $+\infty$ fast enough, so that the series $\sum_na_n[0,\frac{1}{n+1}]$ induces an element in $H_1(M,\germ(K),\A)$ with an infinite norm.
\end{remark}

A subset $X$ of $H_1(M,\germ(K),\A)$ is said \emph{bounded} if for any $K$-connected open neighborhood~$U$ of~$K$, the image of $X$ in $H_1(M,U,\A)$ is bounded, or alternatively, $X\rmod{U}$ is pre-compact.

\begin{lemma}\label{lem-germ-hom-compact}
	Recall $\A=\ZZ$ or $\RR$. Any bounded subset of $H_1(M,\germ(K),\A)$ is pre-compact. 
\end{lemma}

It has the immediate consequence:

\begin{corollary}
	$H_1(M,\germ(K),\A)$ is complete.
\end{corollary}

In particular, $H_1(M,\germ(K),\A)$ can be seen as the completion of the quotient 
$\bfrac{H_1(M,K,\A)}{\ker(\rmod{\germ(K)})}$ for the projective limit topology.

\begin{proof}
	Let $X$ be a bounded subset of $H_1(M,\germ(K),\A)$.
	Since $H_1(M,\germ(K),\A)$ is metrizable, it is enough to prove that any bounded subset is sequentially pre-compact. 
	So let $x_n$ be a sequence in $X$. Up to an extraction, we may assume that for all $k\geq 1$, the image $x_n\rmod{V_{\frac{1}{k}}(K)}$ converges in $n$ toward some $\delta_k$ in $H_1(M,V_{\frac{1}{k}}(K))$. 
	
	For any integers $k\leq l$, we have $\delta_l\rmod{V_{\frac{1}{k}}(K)}=\delta_k$. Thus, the projective limit of the sets $\{\delta_n\}$ is a singleton $\{\delta\}$, which lies in $H_1(M,\germ(K),\A)$. Additionally, $\delta\rmod{V_{\frac{1}{k}}(K)}=\delta_k$ holds for all $k$. It follows that $x_n$ converges toward $\delta$. Therefore, $X$ is pre-compact.
\end{proof}

\begin{lemma}\label{lem-len-semi-cont}
	The simplicial norm $\sell_{\germ(K)}$ is lower semi-continuous.
\end{lemma}

\begin{proof}
	Take some $t$ in $[0,+\infty]$ and a neighborhood~$U$ of~$K$. The simplicial norm $\sell_U$ is lower semi-continuous, so $(\sell_U)^{-1}([0,t])$ is a closed set. By definition, $(\sell_{\germ(K)})^{-1}([0,t])$ is equal to the intersection of the preimages of the sets $(\sell_U)^{-1}([0,t])$, for~$U$ as above. Thus, $(\sell_{\germ(K)})^{-1}([0,t])$ is close.
\end{proof}

\subsection{Cohomology relative to a germ}

We discuss a cohomological construction similar to the homological construction given above.
Given two neighborhoods $U\subset V$ of~$K$, the inclusion $U\hookrightarrow V$ induces a map:
$$H^1(M,V)\xrightarrow{\srmod}H^1(M,U).$$

\begin{definition}
	We define the set:
	$$H^1(M,\germ(K))=\varinjlim_{U\supset K}H^1(M,U)$$
	as a directed limit, where~$U$ varies over the neighborhoods of~$K$, and the maps:
	$$H^1(M,V)\xrightarrow{\srmod}H^1(M,\germ(K))$$
	for any neighborhoods $V$ of~$K$.
\end{definition}

The universal property implies:
$$\bigcup_{\epsilon>0}(H^1(M,V_\epsilon(K))\rmod{\germ(K)})=H^1(M,\germ(K)).$$
So an element in $H^1(M,\germ(K),\RR)$ can be represented has a closed 1-form $\omega$ on~$M$, which is null on some neighborhood of~$K$. Similarly, an element in $H^1(M,\germ(K),\ZZ)$ can be represented has a continuous map $f\colon M\to\bfrac{\RR}{\ZZ}$ which is equal to zero on some neighborhood of $K$.

\begin{proposition}\label{prop-germ-form}
	Recall $\A=\RR$ or $\ZZ$. There is a natural isomorphism of $\A$-modules from $H^1(M,\germ(K),\A)$ to the continuous dual of $H_1(M,\germ(K),\A)$, that is the set of $\A$-linear forms on $H_1(M,\germ(K),\A)$ which are continuous for the projective limit topology.
\end{proposition}

To prove the proposition, we use the following lemma.

\begin{lemma}\label{lem-inj-com-rel}
	Let~$U$ be a $K$-connected neighborhood of~$K$. Then the map $H^1(M,U)\xrightarrow{\srmod} H^1(M,\germ(K))$ is injective.
\end{lemma}

The proof is similar to the one of Lemma~\ref{lem-germ-surf}.

\begin{proof}[Proof of Proposition~\ref{prop-germ-form}]
	We first argue that any~$\beta$ in $H^1(M,\germ(K))$ induces a continuous form on $H_1(M,\germ(K))$. From above follows that we can write $\beta=\alpha\rmod{\germ(K)}$ for some neighborhood $U\supset K$ and some~$\alpha$ in $H^1(M,U)$. Up to choosing~$U$ smaller, we may assume it $K$-connected. Then~$\alpha$ induces a continuous form $H_1(M,U)\xrightarrow{\alpha}\A$. 
	We consider the map $F_\beta\colon H_1(M,\germ(K))\to\A$ obtained as $F_\beta(\delta)=\alpha(\delta\rmod{U})$ for any $\delta$ in $H_1(M,\germ(K))$. It is continuous by definition.
	
	We claim that $F_\beta$ does not depend on the choices of~$U$ and~$\alpha$.
	Indeed, take a second $K$-neighborhood $U'$ of~$K$, and some $\alpha'$ in $H^1(M,U')$ which satisfy $\beta=\alpha'\rmod{\germ(K)}$.
	Choose a $K$-connected neighborhood $V\subset U\cap U'$. By construction, we have:
	$$\alpha\rmod{\germ(K)}=\beta=\alpha'\rmod{\germ(K)}.$$
	It follows from Lemma~\ref{lem-inj-com-rel} that that map $H_1(M,V)\to H_1(M,\germ(K))$ is injective, so
	$\alpha\rmod{V}=\alpha'\rmod{V}$
	holds. Then for any $\delta$ in $H_1(M,\germ(K))$, we have:
	\begin{align*}
		\alpha(\delta\rmod{U})
			&=(\alpha\rmod{V})(\delta\rmod{V}) \\
			&=(\alpha'\rmod{V})(\delta\rmod{V}) \\
			&=\alpha'(\delta\rmod{U'})
	\end{align*}
	It follows that $F_\beta$ is well-defined. With a similar argument, one easily sees that $\beta\mapsto F_\beta$ is $\A$-linear. 

	Let us prove the injectivity. Take~$\beta$ with $F_\beta=0$. As above, we take a $K$-connected neighborhood~$U$ of~$K$, so that $\beta=\alpha\rmod{\germ(K)}$ holds for some~$\alpha$ in $H^1(M,U)$. Then $F_\beta=0$ implies that for any $\delta$ in $H_1(M,\germ(K))$, we have $\alpha(\delta\rmod{U})=0$. It follows from Lemma~\ref{lem-germ-surf} that $H_1(M,\germ(K))\xrightarrow{\srmod} H_1(M,U)$ is surjective, so~$\alpha$ is null. Thus, $\beta=0$ holds.

	We now prove the surjectivity. Let $f\colon H_1(M,\germ(K))\to\A$ be a continuous linear form. By definition of the topology, there exists a $K$-connected neighborhood~$U$ of~$K$, some $C\geq 0$ and a norm $\|\cdot\|$ on $H_1(M,U)$ so that $|f(\delta)|\leq C\|\delta\rmod{U}\|$ holds for any $\delta$ in $H_1(M,\germ(F))$. So $f$ can be factorized as $f(\delta)=f_U(\delta\rmod{U})$ for some linear form $f_U$ on $H_1(M,U)$. Interpreting $f_U$ as an element in $H^1(M,U)$ implies $f=F_{f_U}$. The surjectivity follows.	
\end{proof}

\section{Homological criterion}\label{sec-hom-classification}

In this section, we prove the classification of partial cross-sections using the relative homology. We first introduce the required notions. 

We express a criterion using the module $H_1(M,\germ(\Rec_\alpha),\ZZ)$. When~$\Rec_\alpha$ is empty $H_1(M,\germ(\Rec_\alpha),\ZZ)$ is isomorphic to $H_1(M,\ZZ)$. So some parts of this section are relevant only when~$\Rec_\alpha$ is not empty.

\subsection{Relative asymptotic pseudo-directions}

We define a set of relative asymptotic pseudo-directions which captures some information of the flow that are not captured by the set of asymptotic pseudo-directions. Here we do not require $-\alpha$ to be quasi-Lyapunov.

Take a small $\epsilon>0$. Given an $\epsilon$-pseudo-orbit $\gamma$ that starts and ends on~$\Rec_\alpha$, it has a homology class $[\gamma]$ in $H_1(M,\Rec_\alpha,\ZZ)$.
Denote by $D_{\varphi,\alpha,\epsilon}$ the subset of $H_1(M,\germ(\Rec_\alpha),\ZZ)$ made of the classes $[\gamma]\rmod{\germ(\Rec_\alpha)}$ for these $\epsilon$-pseudo-orbits~$\gamma$. Also denote by $\wb{D_{\varphi,\alpha,\epsilon}}$ its closure.
\begin{definition}
	We call the set of \emph{$\alpha$-relative asymptotic pseudo-directions} the set $D_{\varphi,\alpha}$ defined as 
	$$D_{\varphi,\alpha}=\bigcap_{\epsilon>0}\wb{D_{\varphi,\alpha,\epsilon}}\subset H_1(M,\germ(\Rec_\alpha),\ZZ),$$
	which does not depend on $T$.
\end{definition}

Contrary to $D_\varphi$, the set $D_{\varphi,\alpha}$ is not necessarily bounded and Thus, not necessarily compact.


\begin{lemma}\label{lem-D-len-1}
	$D_{\varphi,\alpha}$ lies in $(\sell_{\germ(\Rec_\alpha)})^{-1}(\{0,1\})$. 
\end{lemma}

Together wise Proposition~\ref{prop-img-hom-finit-l}, it has the consequence:

\begin{corollary}
	$D_{\varphi,\alpha}$ lies in $H_1(M,\Rec_\alpha,\ZZ)\rmod{\germ(\Rec_\alpha)}$. 
\end{corollary}

So any~$\alpha$-relative asymptotic directions is representable by a brave curve with boundary in~$\Rec_\alpha$. 

\begin{proof}[Proof of Lemma~\ref{lem-D-len-1}]
	From definition, $D_{\varphi,\alpha,\epsilon}$ lies in $(\sell_{\germ(\Rec_\alpha)})^{-1}(\{0,1\})$. It follows from the lower semi-continuity of $\sell_{\germ(\Rec_\alpha)}$ (see Lemma~\ref{lem-len-semi-cont}) that $\wb{D_{\varphi,\alpha,\epsilon}}$ lies in $(\sell_{\germ(\Rec_\alpha)})^{-1}(\{0,1\}))$, and so does $D_{\varphi,\alpha}$.
\end{proof}

Contrary to asymptotic pseudo-directions, we can not use a compactness argument to generate asymptotic pseudo-directions. Instead, we rely on the next result.

Let $f\colon\wh M_\alpha\to\RR$ be an~$\alpha$-equivariant continuous map. 
Call a \emph{$C$-thick slice} (for $f$) the pre-image $f^{-1}(I)$ of any interval $I\subset\RR$ of length $C$.
Take $\epsilon>0$ small. An $\epsilon$-pseudo-orbit $\gamma$ in $M$ is said to remain in a $C$-thick slice if, given any $\epsilon$-pseudo-orbit $\wh\gamma$ that lifts $\gamma$ in $\wh M_\alpha$, $\wh\gamma$ remains in a $C$-thick slice.

\begin{proposition}\label{prop-rel-ass-dir}
	Fix~$\alpha$ in $H^1(M,\ZZ)$ and an~$\alpha$-equivariant map $f\colon\wh M_\alpha\to\RR$. For any $C\geq 0$ and any~$\alpha$-recurrence chains $R_0$, and $R_1$, there exists $\epsilon>0$ which satisfies the following. For any $\epsilon$-pseudo-orbit $\gamma$ from $R_0$ to $R_1$, which remains in a $C$-thick slice, the homology class $[\gamma]\rmod{\germ(\Rec_\alpha)}$ lies in $D_{\varphi,\alpha}$.
\end{proposition}

Note that we do not require $-\alpha$ to be quasi-Lyapunov. But when $-\alpha$ is quasi-Lyapunov, the condition that $\gamma$ remains in a $C$-thick slice can be replaced by $\int_\gamma df\leq C$, maybe for a different value of $C$.

To prove the proposition, we first control the behavior of $\gamma$ outside a small neighborhood of~$\Rec_\alpha$, then prove the proposition modulo a small open neighborhood of~$\Rec_\alpha$.

\begin{lemma}\label{lem-neg-pso-acc-rec}
	Take $C\geq 0$ and a neighborhood~$U$ of~$\Rec_\alpha$. There exists $n$ and $\epsilon>0$ which satisfies the following. For any $\epsilon$-pseudo-orbit $\gamma$ in $M$ that remains in a $C$-thick slice of $M$, there exists a union of $n$ arcs of $\gamma$, each of length $T$, so that $\gamma$ remains inside~$U$ outside these arcs.
\end{lemma}

\begin{corollary}
	In Lemma~\ref{lem-neg-pso-acc-rec}, when $\epsilon$ goes to zero and $\len(\gamma)$ goes to $+\infty$, some point in $\gamma$ accumulates on~$\Rec_\alpha$. 
\end{corollary}

\begin{proof}
	Choose a smaller neighborhood $V\subset U$ of~$\Rec_\alpha$ so that $\varphi_{[-2T,2T]}(V)\subset U$ holds. Take $\epsilon>0$ very small and an $\epsilon$-pseudo-orbit $\gamma$ that remains in a $C$-thick slice. 

	We reason by contradiction.
	Take some $n$ large and assume that $\gamma$ has at least $n$ times $t_1\cdots t_n$, with $t_{k+1}>t_k+2T$, so that $\gamma$ is continuous on $[t_k,t_k+T[$, and so that $\gamma([t_k,t_k+T[)$ intersects $M\setminus U$.
	Note that the arc $\gamma([t_kt,_k+T[)$ remains outside $V$. So if $n$ is large enough, there exist two indexes $1\leq k<k'\leq n$ that satisfy the following. Firstly, $\gamma(t_k)$ and $\gamma(t_{k'})$ are $\epsilon$-close. Secondly, the restriction $\gamma([t_k,t_{k'}])$ induces a periodic $\epsilon$-pseudo-orbit which satisfies $\alpha([\gamma([t_k,t_{k'}])])=0$. This is made possible by lifting $\gamma$ to~$\wh M_\alpha$, and considering the accumulation points at the times $t_k$ inside a given $C$-thick slice.

	So take $k,k'$ as above. When $\epsilon$ goes to zero, $\gamma(t_k)$ accumulates on the set~$\Rec_\alpha$. Thus, when $\epsilon$ is small enough, $\gamma(t_k)$ lies in $V$, which contradicts the assumption above.
\end{proof}

\begin{lemma}\label{lem-rel-ass-dir-tool}
	Assume $\alpha(D_\varphi)\geq0$. For any $C\geq 0$ and any open neighborhood~$U$ of~$\Rec_\alpha$, there exists $\epsilon>0$ which satisfies the following. For any $\epsilon$-pseudo-orbit $\gamma$ in~$M$ that remains in a $C$-thick slice, and with ends on~$\Rec_\alpha$, the homology class of $[\gamma]\rmod{U}$ in $H_1(M,U,\ZZ)$ lies in $D_{\varphi,\alpha}\rmod{U}$.
\end{lemma}

Stated differently, the cohomology classes $[\gamma]\rmod{\germ(\Rec_\alpha)}$ accumulates on $D_{\varphi,\alpha}$ when $\epsilon$ goes to zero, under the constraints $\partial\gamma\subset\Rec_\alpha$ and that $\gamma$ remains in a $C$-thick slice.

\begin{proof}
	Let $f$, $C\geq 0$ and~$U$ be as in the lemma. Up to making it smaller, we may assume that~$U$ is~$\Rec_\alpha$-connected.

	Take $\epsilon>0$ and let $\gamma$ be an $\epsilon$-pseudo-orbit in $M$, that remains in a $C$-thick slice.
	From Lemma~\ref{lem-neg-pso-acc-rec}, $\gamma$ spends a bounded above time outside~$U$. So	the homology classes $[\gamma]\rmod{U}$ in $H_1(M,U,\ZZ)$ belongs to a bounded set, that depends on $\epsilon$ but not $\gamma$.

	Denote by $X_\epsilon$ the closure in $H_1(M,\germ(\Rec_\alpha),\ZZ)$ of the set of $[\gamma]\rmod{\germ(\Rec_\alpha)}$ where $\gamma$ varies among the $\epsilon$-pseudo-orbits considered above.
	It follows from above that for any neighborhood $V$ of~$\Rec_\alpha$, the image $X_\epsilon\rmod{V}$ is finite. So~$X_\epsilon$ is bounded and closed.
	By Lemma~\ref{lem-germ-hom-compact}, the set $X_\epsilon$ is compact. Denote by $X$ the monotonous intersection of the compact $X_\epsilon$, necessarily compact and non-empty. Note that $X_\epsilon$ lies in $D_{\varphi,\alpha,\epsilon}$, so $X$ lies inside $D_{\varphi,\alpha}$. 

	For any~$\Rec_\alpha$-connected neighborhood~$U$ of~$\Rec_\alpha$, the image $X_\epsilon\rmod{U}$ in $H_1(M,U,\ZZ)$ is finite and non-increasing when $\epsilon$ decreases. So it stabilizes, and it is equal to $X\rmod{U}$ when $\epsilon$ is small enough. Then, taking $\epsilon$ small enough, when~$\gamma$ is an $\epsilon$-pseudo-orbit as above, its cohomology class $[\gamma]\rmod{U}$ lies in $X_\epsilon\rmod{U}=X\rmod{U}\subset D_{\varphi,\alpha}\rmod{U}$.
\end{proof}

\begin{proof}[Proof of Proposition~\ref{prop-rel-ass-dir}]
	Let $f\colon\wh M_\alpha\to\RR$ be continuous and~$\alpha$-equivariant, $R_0$ and $R_1$ be two~$\alpha$-recurrence chains, and take some $C\geq 0$. Let $\epsilon_n>0$ be a sequence converging toward zero, and $\gamma_n$ be a $(\epsilon_n,T)$-pseudo-orbit from $R_0$ to~$R_1$, and that remains in a $C$-thick slice. We claim that $[\gamma_n]\rmod{\germ(\Rec_\alpha)}$ is constant for large enough $n$, and that it lies in $D_{\varphi,\alpha}$. The conclusion follows. 

	From Lemma~\ref{lem-rel-ass-dir-tool}, the classes $[\gamma_n]\rmod{\germ(\Rec_\alpha)}$ converges toward some $\delta$ in $H_1(M,\germ(\Rec_\alpha),\ZZ)$. By assumption, the boundary of $[\gamma_n]-[\gamma_m]$ is null. So there exists some $x_{n,m}$ in $H_1(M,\ZZ)$ with $x\rmod{\Rec_\alpha}=[\gamma_n]-[\gamma_m]$. We have:
	\begin{equation}\label{eq-conv-tool}
		x_{n,m}\rmod{\germ(\Rec_\alpha)}=([\gamma_n]-[\gamma_m])\rmod{\germ(\Rec_\alpha)}
	\end{equation}
	so when $n$ and $m$ converge toward $+\infty$, $x_{n,m}\rmod{\germ(\Rec_\alpha)}$ converges toward zero. The submodule $H_1(M,\ZZ)\rmod{\germ(\Rec_\alpha)}$ of $H_1(M,\germ(\Rec_\alpha),\ZZ)$ is finitely generated, so it is discrete. Hence, when $n$ and $m$ are large enough, the image $x_{n,m}\rmod{\germ(\Rec_\alpha)}$ is equal to zero. The claim follows from Equation~\ref{eq-conv-tool}.
\end{proof}

\subsection{Classification in relative homology}\label{sec-proof-homo-criterion}

Let~$\alpha$ be in $H^1(M,\ZZ)$, so that $-\alpha$ is quasi-Lyapunov, and let~$S$ be a partial cross-section cohomologous to~$\alpha$. We justify that~$S$ induces a relative cohomology class $[S]$ inside $H^1(M,\germ(\Rec_\alpha),\ZZ)$. Consider the map $F_S\colon M\to\bfrac{\RR}{\ZZ}$ defined in Section~\ref{sec-partial-section}. Recall $-F_S$ that it is pre-Lyapunov, and it is equal to $0$ on a neighborhood~$\Rec_\alpha$. Hence, there is an open neighborhood~$U$ of~$\Rec_\alpha$ which satisfies $F_S(U)=0$. So $F_S$ has a cohomology class in $H^1(M,U,\ZZ)$, which has an image in $H^1(M,\germ(\Rec_\alpha),\ZZ)$.

\begin{definition}
	We denote by $[S]_\germ=[F_S]_\germ$ the relative cohomology class of $F_S$ in $H^1(M,\germ(\Rec_\alpha),\ZZ)$.
\end{definition}


Consider the natural map $H^1(M,\germ(\Rec_\alpha),\ZZ)\xrightarrow{\rmod{\emptyset}} H^1(M,\ZZ)$.

\begin{theorem}[Restatement of Theorem~\ref{mainthm-classification}]\label{thm-ps-classification-3}
	Assume $M$ compact and connected. Let~$\alpha$ be in $H^1(M,\ZZ)$. If $-\alpha$ is quasi-Lyapunov, then the map 
	$$\begin{tikzcd}
		\PS_\varphi(\alpha) \arrow[r] & H^1(M,\germ(\Rec_\alpha),\ZZ)\setminus\{0\} \\ [-0.9cm]
		S \arrow[r, mapsto] & {[S]_\germ}
	\end{tikzcd}$$
	induces a bijection onto the set of non-zero classes~$\beta$ in $H^1(M,\germ(\Rec_\alpha),\ZZ)$ which satisfies $\beta\rmod{\emptyset} = \alpha$ and $\beta(D_{\varphi,\alpha})\geq 0$.
\end{theorem}

When~$\alpha$ is not zero and $\beta\rmod{\emptyset}=\alpha$ holds, we automatically have $\beta\neq 0$.
The condition $\beta\neq 0$ is present for the case $\alpha=0$, where $\beta=0$ does not represent any null-cohomologous partial cross-section. 

Note that $-\alpha$ is automatically quasi-Lyapunov in the case $\alpha=0$. When~$\alpha$ is not zero, one may want to replace the quasi-Lyapunov assumption with a weaker hypothesis, for instance $\alpha(D_\varphi)\geq 0$. We put the emphasis on the fact that it is not possible in general. The sets $D_\varphi$ and $D_{\varphi,\alpha}$ do not contain enough information to guarantee the quasi-Lyapunov property.
In Section~\ref{sec-rem-ass}, we give two counter-examples.

\begin{proof}
	We first prove that $[S]_\germ$ satisfies the same properties as~$\beta$ in the conclusion of the theorem. The class $[S]_\germ$ is invariant by isotopy along the flow of~$S$, so it depends only on the image of~$S$ inside $\PS_\varphi(\alpha)$. 
	
	Denote by $\pi_\alpha\colon\wh M_\alpha\to M$ the covering map, and by $\wh F_S\colon\wh M_\alpha\to\RR$ an $\alpha$-equivariant lift of $F_S$.
	For any recurrence chains $R_1$ and $R_2$ in~$\wh M_\alpha$, observe the equality $\wh F_{S}(R_2)-\wh F_S(R_1)=[S]_\germ([\pi_\alpha(\delta)]\rmod{\germ(\Rec_\alpha)})$, where $\delta$ is any curve that starts in $R_1$ and ends on $R_2$, 

	Since $-F_S$ is pre-Lyapunov, it follows $[S]_\germ(D_{\varphi,\alpha})\geq0$ from the remark above. By construction, $[S]_\germ$ is the image of the class of $F_S$ in $H^1(M,U,\ZZ)$ for some open neighborhood~$U$ of~$\Rec_\alpha$. The image of that class in $H^1(M,\ZZ)$ is~$\alpha$ by construction of $F_S$, so $[S]_s\rmod{\emptyset}$ is equal to~$\alpha$.

	We prove here the injectivity. Let $S_1,S_2$ be two partial cross-sections cohomologous to~$\alpha$, with $[S_1]_\germ=[S_2]_\germ$. Denote by $\wh F_{S_1},\wh F_{S_2}\colon \wh M_\alpha\to\RR$ the maps they induce.
	For any recurrence chains $R_1,R_2$ in~$\wh M_\alpha$, and any curve $\delta$ from $R_1$ to $R_2$, the quantity $\wh F_{S_i}(R_2)-\wh F_{S_i}(R_1)$ is equal to $[S_i]_\germ([\pi_\alpha(\delta)]\rmod{\germ(\Rec_\alpha)})$. So $\wh F_{S_i}(R_2)-\wh F_{S_i}(R_1)$ is independent on $i$. 
	It follows that the quantity $\wh F_{S_1}(R)-\wh F_{S_2}(R)$ is constant when $R$ varies over the recurrence chains in~$\wh M_\alpha$. Thus, the images of $\wh F_{S_1}$ and $\wh F_{S_2}$ in $I_\alpha(\wh\Rec_\alpha^*,\ZZ)$ are equal. By Theorem~\ref{thm-preL-to-order}, $S_1$ and $S_2$ are isotopic along the flow.
	
	Let us prove that $\PS_\varphi(\alpha)$ is not empty.	
	Assume $\alpha=0$. If there exists~$\beta$ in $H^1(M,\germ(\Rec_\alpha),\ZZ)$ that is not zero, then $H^1(M,\germ(\Rec_\alpha),\ZZ)$ is not trivial and~$\Rec_\alpha$ is strictly smaller than $M$. It follows that $\varphi$ is not chain recurrent, and from Theorem~\ref{thm-ps-to-non-rec}, $\PS_\alpha(\varphi)$ is not empty.
	In the case $\alpha\neq0$, Theorem~\ref{thm-ps-to-subL} also implies that $\PS_\alpha(\varphi)$ is not empty. 
	
	We now prove the surjectivity. 
	Take~$\beta$ in $ H^1(M,\germ(\Rec_\alpha),\ZZ)$ which satisfies $\beta(D_{\varphi,\alpha})\geq 0$ and $\beta\rmod{\emptyset}=\alpha$. We define a map $\wh f$ in $\wh I_\alpha(\wh\Rec_\alpha^*,\ZZ)$ from~$\beta$. Fix a based point $x_0$ in $\wh\Rec_\alpha\subset\wh M_\alpha$ and define $\wh f(x_0)=0$. Given any recurrent point $y$ in $\wh\Rec_\alpha$, take a path $c$ from $x_0$ to $y$, and set $\wh f(y)=\beta([\pi_\alpha(c)]\rmod{\germ(\Rec_\alpha)})$, where $[\pi_\alpha(c)]$ is the homology class of $\pi_\alpha(c)$ in $H_1(M,\Rec_\alpha,\ZZ)$. 
	
	The hypothesis $\beta\rmod{\emptyset}=\alpha$ implies that $\beta([\pi_\alpha(c)]\rmod{\germ(\Rec_\alpha)})$ does not depend on the choice of path $c$. Indeed, if $c'$ is another path from $x_0$ to~$y$, one has
	$$\beta([\pi_\alpha(c'-c)]\rmod{\germ(\Rec_\alpha)})=\alpha([\pi_\alpha(c'-c)]\rmod{\germ(\Rec_\alpha)})=0,$$
	since $c'-c$ is a closed curve inside $\wh M_\alpha$. So $\wh f$ is well-defined. A similar argument shows that $\wh f$ is~$\alpha$-equivariant. The assumption $\beta\neq 0$ implies that $\wh f$ is not constant.
	
	We now show that $-\wh f$ is pre-Lyapunov.
	Take two recurrent points $x,y$ with $x\recto y$, some $\epsilon>0$ small and an $\epsilon$-pseudo-orbit $\gamma$ from $x$ to $y$. 
	Let use choose an~$\alpha$-equivariant and quasi-Lyapunov map $g\colon\wh M_\alpha\to\RR$. Since the endpoints of $\gamma$ are fixed, $\gamma$ remains in a $C$-thick slice of $g$, for some $C$ that depend on $x$, $y$ and $g$ only.

	By Proposition~\ref{prop-rel-ass-dir}, when $\epsilon$ is small enough, the relative homology class of $\pi_\alpha(\gamma)$ lies inside $D_{\varphi,\alpha}$. 
	Thus, $\wh f(y)-\wh f(x)=\beta([\pi_\alpha(\gamma)]\rmod{\germ(\Rec_\alpha)})\geq 0$ holds. It follows $\wh f(y)\geq \wh f(x)$. 
	When $x$ and $y$ are additionally on the same recurrence chain, then $\wh f(y)=\wh f(x)$ holds. So $-\wh f$ is pre-Lyapunov.

	The class~$\beta$ has integer coefficients, so the image of $\wh f$ lies inside $\ZZ$. 
	Therefore, $\wh f$ induces a map $\wh f^*\colon\wh\Rec_\alpha^*\to\ZZ$.
	It is continuous by continuity of~$\beta$. More precisely, there exists a neighborhood~$U$ of~$\Rec_\alpha$ so that~$\beta$ lies in the image of $H^1(M,U,\ZZ)$. Then any two recurrence chains that are close enough lies in the same path-connected component of~$U$, and so have the same value by $\wh f^*$.
	
	It follows from above that $\wh f^*$ is~$\alpha$-equivariant and non-decreasing.
	Using Theorem~\ref{thm-preL-to-order}, $\wh f^*$ represents a partial cross-section~$S$. By construction, we have $[S]_\germ=\beta$, which concludes the proof. 
\end{proof}

Note in the last part of the proof that it is required to know that $\PS_\varphi(\alpha)$ is not empty, otherwise we can not apply Theorem~\ref{thm-preL-to-order}.

\subsection{Necessity of one hypothesis}\label{sec-rem-ass}

In this section, we discuss the necessity of assuming $-\alpha$ quasi-Lyapunov in Theorem~\ref{thm-ps-classification-3}. One could hope replacing it with simply $\alpha(D_\varphi)\geq 0$, but this is not a sufficient condition. More precisely, there exist examples of where $\alpha(D_\varphi)\geq 0$ holds, where there exists~$\beta$ in $H_1(M,\germ(\Rec_\alpha),\ZZ)$ with $\beta\rmod{\emptyset}=\alpha$ and $\beta(D_{\varphi,\alpha})\geq0$, but where $-\alpha$ is not quasi-Lyapunov. 
We give a first example where $D_\varphi$ and $D_{\varphi,\alpha}$ are equal to $\{0\}$, and a second example where it is not the case.

Fix a vector $(a,b)$ in $\RR^2$.
Take $\TT^2=\bfrac{\RR^2}{\ZZ^2}$ and let $\psi_1$ be the flow on $\TT^2$ generated by the vector field $(a,b)\lambda$, for some continuous function $\lambda\colon \TT^2\to[0,1]$ that is equal to zero on exactly one point $p_0$. 

\begin{lemma}\label{lem-compute-D1}
	If $(a,b)$ has an irrational slope and if $\lambda$ converges fast enough to zero on $p_0$, then $D_{\psi_1}=\{0\}$ holds.	
\end{lemma}

Given a pseudo-orbit, we call the \emph{physical length} of $\gamma$, integral of $\|\gamma'\|$. That is the length of the corresponding curve and not the length of the parametrization as a pseudo-orbit. 

\begin{proof}
	Assume that $(a,b)$ has an irrational slope.
	Denote by $B_r$ the ball of radius $r$ around $p_0$. Take $\epsilon>0$ and a periodic $\epsilon$-pseudo-orbit $\gamma$ with $[\gamma]\neq0$.

	Denote by $L$ the physical length of $\gamma$, and by $L_r$ the physical length of $\gamma\cap B_r$. 
	Assume that $\gamma$ has a minimal amount of jumps among the $\epsilon$-$\epsilon$-pseudo-orbit with that homology class. So that, since $\psi_1$ has an irrational slope, the image of $\gamma$ fills the torus in an almost uniform manner. 

	When $\epsilon$ is small enough, the homology class $[\gamma]$ is comparable to $(a,b)L$. More precisely $\frac{1}{L}[\gamma]$ converges toward $(a,b)$ when $\epsilon$ goes to zero.
	Similarly, $L_r$ is comparable to $L$ times the Lebesgue measure $\Leb(B_r)$ of $B_r$, that is $\frac{L_r}{L}$ converges toward $\Leb(B_r)$.
	
	Denote by $\lambda_r=\sup\lambda(B_r)$, it follows from above that when $\epsilon$ is small enough, the length $\len(\gamma)$ of $\gamma$ is bounded below by 
	$$\len(\gamma)\geq\frac{L_r}{2\lambda_r}\approx\frac{Leb(B_r)L}{2\lambda_r}\approx\frac{Leb(B_r)}{2\lambda_r}\frac{\|[\gamma]\|}{\|(a,b)\|}$$
	Thus, the quotient $\frac{1}{\len(\gamma)}[\gamma]$ has a norm that is bounded above by a constant times $\frac{\lambda_r}{Leb(B_r)}$. It follows that if $\frac{\lambda_r}{\Leb(B_r)}$ goes to zero when $r$ goes to zero, then $\frac{1}{\len(\gamma)}[\gamma]$ goes to zero when $\epsilon$ goes to zero. Then, $D_{\psi_1}$ is equal to $\{0\}$.
\end{proof}

We now assume $D_{\psi_1}=\{0\}$. Take $a<0$ and consider the class $\alpha=dx$, which satisfies $\alpha(D_{\psi_1})\geq0$. 
Clearly, $-\alpha$ is not Lyapunov, and $\PS_\varphi(\alpha)$ is empty.
Note that the~$\alpha$-recurrent set is $\Rec_\alpha=\{p_0\}$, so we may identify $H_1(\TT^2,\germ(\Rec_\alpha),\ZZ)$ with $H_1(\TT^2,p_0,\ZZ)$. 

We identify $\wh \TT^2_\alpha$ with $\RR\times(\bfrac{\RR}{\ZZ})$. Fix an $\alpha$-equivariant map $f\colon\wh \TT^2_\alpha\to\RR$, defined by $f(x,y)=x$. 
Note that the map that sends $\delta$ in $H_1(\TT^2,\germ(\Rec_\alpha),\ZZ)$ to $\int_\delta df$ is well-defined and induces an element~$\beta$ in $H^1(\TT^2,\germ(\Rec_\alpha),\ZZ)$, which satisfies $\beta\rmod{\emptyset}=\alpha$.	

\begin{lemma}\label{lem-compute-D2}
	If $(a,b)$ has an irrational slope, then $D_{\psi_1,\alpha}=\{0\}$ holds.
\end{lemma}

It has the immediate consequence.

\begin{corollary}
	Assume that $(a,b)$ has an irrational slope, that $\lambda$ converges fast enough to zero in $p_0$, and $a<0$. Then for $\alpha=dx$, we have $\alpha(D_{\psi_1})\geq 0$, and there exists~$\beta$ in $H_1(\TT^2,\germ(\Rec_\alpha),\ZZ)$, with $\beta\rmod{\emptyset}=\alpha$ and $\beta(D_{{\psi_1},\alpha})\geq 0$, that does not represent any partial cross-section.
\end{corollary}

\begin{proof}
	Assume that there exists a non-zero element $\delta$ in $D_{\varphi,\alpha}$. It can be approximated by the homology class of an $\epsilon$-pseudo-orbit $\gamma$ from $p_0$ to itself. When $\epsilon$ is small enough, $\gamma$ must be very close to a curve of slope $(a,b)$. And since the flow has irrational slope, and since $[\gamma]$ is not zero, the physical length of $\gamma$ must be very large. Then $\int_\gamma df$ is close to the physical length of $\gamma$ times $\frac{a}{a^2+b^2}$, so it is larger than $\beta(\delta)$ in absolute value. It contradicts that $\beta(\delta)=\int_\gamma df$ is finite.
\end{proof}

It may seem like the counter-example rely on the fact that $D_\varphi$, $D_{\varphi,\alpha}$ are zero. We give a second example, where the two sets above are not trivial.

Take $\TT^3=\bfrac{\RR^3}{\ZZ^3}$. Let us build a flow $\psi_2$ by pieces. We equip $\TT^2\times\{0\}$ with a copy of the flow $\psi_1$ using $a<0$, $(a,b)$ with an irrational slope, and $\lambda$ going fast enough to zero on $p_0$. We equip $\TT^2\times\{\frac{1}{2}\}$ with another copy of $\psi_1$, with $(a,b)=(1,0)$. 
We equip $\TT^2\times\{\frac{1}{4},\frac{3}{4}\}$ by linear flows directed by $(1,0)$.
Then we extend the flow outside $\TT^2\times\{0,\frac{1}{4},\frac{1}{2},\frac{3}{4}\}$ by orbits that accumulates on $\TT^2\times\{0,\tfrac{1}{2}\}$ in the past and on $\TT^2\times\{\frac{1}{4},\frac{3}{4}\}$ in the future. 

As above, we take $\alpha=dx$, $f(x,y,z)=x$ and~$\beta$ in $H^1(\TT^3,\germ(\Rec_\alpha),\ZZ)$ that represent $\int_\cdot df$.

\begin{lemma}
	We have $\alpha(D_{\psi_2})\gneq0$ and $\beta(D_{\psi_2,\alpha})\gneq0$.
\end{lemma}

It immediately follows:

\begin{corollary}
	We have $\alpha(D_{\psi_2})\gneq 0$, and there exists~$\beta$ in $H_1(\TT^3,\germ(\Rec_\alpha),\ZZ)$ with $\beta\rmod{\emptyset}=\alpha$ and $\beta(D_{\psi_2,\alpha})\gneq 0$, and that does not represent a partial cross-section.
\end{corollary}

\begin{proof}
	Take $\epsilon$ very small.
	Any periodic $\epsilon$-pseudo-orbit remains very close to a torus $\TT^2\times\{t\}$ for $t$ in $\{0,\frac{1}{4},\frac{1}{2},\frac{3}{4}\}$. 
	
	By Lemma~\ref{lem-compute-D1}, the contribution of $\TT^2\times\{0\}$ to $D_{\psi_2}$ is $\{0\}$. The contributions of $\TT^2\times\{t\}$ to $D_{\psi_2}$, for $t$ in $\{\frac{1}{4},\frac{1}{2},\frac{3}{4}\}$, clearly do not have negative values by~$\alpha$. Thus, we have $\alpha(D_{\psi_2})\geq 0$. 

	We write $\delta=\bfrac{\RR}{\ZZ}\times(0,\frac{1}{4})$. 
	Notice that $\delta$ is a periodic orbit of $\psi_2$, of period one. So $D_{\psi_2}$ contains $[\delta]$, which satisfies $\alpha([\delta])=1$. It follows $\alpha(D_{\psi_2})\gneq0$.

	We now prove the similar identity for $\beta$.
	Denote by $p_0=(x_0,y_0,0)$ and $p_1=(x_1,y_1,\tfrac{1}{2})$ the two fixed points of $\psi_2$. 
	There is no periodic $\epsilon$-pseudo-orbit $\gamma$ that satisfies $\alpha(\gamma)=0$ and that get out of a small neighborhood of $\{p_0,p_1\}$. It follows $\Rec_\alpha=\{p_0,p_1\}$. 

	Take a periodic $\epsilon$-pseudo-orbit~$\gamma$ that starts and ends on~$\Rec_\alpha$. 
	Note that $\TT^2\times\{\frac{1}{4},\frac{3}{4}\}$ is attracting, so it can not be crossed by $\epsilon$-pseudo-orbit that goes from $\Rec_\alpha$ to itself.

	Assume that $\gamma$ starts on $p_0$. From above, $\gamma$ also ends on $p_0$, and it remains very close to the torus $\TT^2\times\{0\}$. Thus, it is shadowed by an $\epsilon$-pseudo-orbit that remains inside that torus. 
	It follows from Lemma~\ref{lem-compute-D2} that $\gamma$ does not lie in $D_{\psi_2,\alpha}\setminus\{0\}$. So the contribution of $\TT^2\times\{0\}$ to $D_{\psi_2,\alpha}$ is $\{0\}$. 
	
	Similarly, if $\gamma$ starts on $p_1$, it ends on $p_1$ and $\beta(\gamma)=\alpha(\gamma)\geq 0$ holds (as long as $\epsilon$ is small enough). So the contribution of $\TT^2\times\{\frac{1}{2}\}$ to $D_{\psi_2,\alpha}$ has non-negative values of~$\beta$. Therefore, $\beta(D_{\psi_2,\alpha})\geq0$ holds.
	
	We take $\delta'=\bfrac{\RR}{\ZZ}\times(y_1,\frac{1}{2})$. We can interpret $\delta'\setminus\{p_1\}$ as an $\epsilon$-pseudo-orbit (for any $\epsilon>0$) that starts and ends on $p_1$. So its homology class belongs to~$D_{\psi_2,\alpha}$. We have $\beta(\delta')=1$. So $\beta(D_{\psi_2,\alpha})\gneq0$ holds.
\end{proof}

\begin{remark}
	These flows can be made non-singular, that is non-zero everywhere, by taking their products with a non-singular flow on $\bfrac{\RR}{\ZZ}$. So we can not hope for a better assumption for non-singular flows.
\end{remark}

\appendix

\section{Smoothing a partial cross-section}\label{sec-smoothing}

In this paper, we describe partial cross-sections without any regularity assumption, so that the set of partial cross-sections is invariant under continuous conjugation. Let us describe a reason to consider the partial cross-sections that are only continuous.

\begin{remark}
	Given a $\Class^1$ flows, and a periodic orbit $\gamma$, one may want to consider the blown up manifold $M_\gamma$ at $\gamma$. The lifted flow on $M_\gamma$ is a priori only continuous. Then, one may want to blow down $M_\gamma$ in a Dehn-surgery manner, obtaining again a flow that is a priori only continuous. In both case, our theorems remain valid for these flows.
\end{remark}

Since it is more common to work with smooth partial cross-sections, we need an approximation result.

\begin{lemma}
	Assume $M$ smooth.
	Let~$S$ be a partial cross-section and assume that, on a neighborhood of~$S$, $\varphi$ is obtained by integrating a continuous and uniquely integrable vector field $X$. Then for any $\epsilon>0$, there exist a smooth partial cross-section $S'$, transverse to $X$, and an isotopy along the flow from~$S$ to $S'$, that remains in the $\epsilon$-neighborhood of~$S$.
\end{lemma}

As discuss below, there exists wild continuous flows for which some partial cross-sections are not approximable by smooth partial cross-sections. 

\begin{proof}
	Similarly to the proof of Lemma~\ref{lem-loc-flat}, take $d>0$ and $f\colon S\times[-d,d]\to M$ a parametrization of a neighborhood of~$S$. 
	Denote by~$\alpha$ the cohomology class of~$S$. Note that~$S$ splits $\wh M_\alpha$ in two regions, $U_+$ in the positive side of~$S$ (for the flow) and $U_-$ in the negative side. That is every point $x$ in $U_\pm$ is in the same connected component of $\wh M_\alpha\setminus S$ than some point $f(y,\pm d)$.
	
	We define the function $h\colon\wh M_\alpha\to\RR$ by	$h\circ f(x,t)=-t$, $h(U^+)=-1$ and $h(U^-)=1$. It is non-decreasing along the flow.
	So from \cite[Theorem 2.1]{Fathi2019}, we can approximate $h$ with a smooth function $g$, which satisfies $dg(X)<0$ outside $U_+\cup U_-$. We set $S'=g^{-1}(0)$. Assume that $g$ is close enough to $h$, so that $S'$ is disjoint from $U_+\cup U_-$. Then $S'$ is a smooth, compact, hypersurface transverse to $X$. In particular $S'$ is a partial cross-section of $\varphi$. 
	
	Consider the homotopy of maps $g_t=tg+(1-t)h$, from $g_0=h$ to $g_1=g$, and $S_t=g_t^{-1}(\{0\})$. It is clear that $g_t$ is increasing along the flow close to $S_t$. So $S_t$ is an isotopy of partial cross-sections between $S_0=S$ and $S_1=S'$. From above, it is an isotopy along the flow. When $g$ is very close to $h$, $S_t$ remains in the $\epsilon$-neighborhood of~$S$.
\end{proof}

We give an example of a continuous flow $\varphi$ on $M=[-1,1]\times\RR$ which admits a partial cross-section that is not approximable by $\Class^1$ partial-sections, even only topologically transverse. Consider first the strip $N=[0,1]\times\RR$ equip with a $\Class^1$ flow $\psi$ that satisfies the following:
\begin{itemize}
	\item $\psi$ is equivariant by the action of $\ZZ$ on $N$ by vertical translation,
	\item $\varphi_t(x,y)=(x,y+t)$ for $x$ in $\{0,1\}$, and any $y,t$ in $\RR$,
	\item $\varphi$ is parametrized to go asymptotically upward at speed one, that is there exists $c$ so that for all $x,y,t,x',y',$ with $\varphi_t(x,y)=(x',y')$, we have $y'\geq y+t-c$,
	\item there exists an orbit $\gamma_0$ that, for all $n$ in $\ZZ$, passes through $(\tfrac{3}{4},n)$, then $(\tfrac{1}{4},n+3)$, then through $(\tfrac{3}{4},n+1)$ and so on.
\end{itemize}

The last item ensures that for any $\Class^1$ curve $\delta$ so that $\frac{\partial\delta(t)}{\partial t}$ is close enough to $(1,b)$, for some $b\geq 0$, $\delta$ intersects $\gamma_0$ transversally in at least two points, with the two possible co-orientations. In particular $\delta$ is not topologically transverse to $\varphi$.

We place in $M$ infinitely many copies of $N$. For $n\geq 1$ and $\sigma$ in $\{+,-\}$, consider the map $f^\sigma_n$ given by:
$$\begin{tikzcd}
	|[xshift=-0.9em]| f^\sigma_n\colon N \arrow[r,"\simeq"] & |[xshift=0.7em]| (\sigma[\frac{1}{3^n},\frac{1}{3^{n-1}}])\times\RR\subset M \\ [-0.9cm]
	(x,y) \arrow[r, mapsto] & \left(\sigma\frac{2x+1}{3^n},\frac{y}{2^n}\right)\!.
\end{tikzcd}$$
Note that their images cover $M$ minus $\{0\}\times\RR$. 

We define the flow $\varphi$ on $M$ as follows. On the image of $f^\sigma_n$, $\varphi$ coincides with the push-forward of $\psi$, reparametrization by a speed factor $2^n$ to go asymptotically upward at speed one. We also define $\varphi$ on $\{0\}\times\RR$ to go upward at constant speed one. 
We let the reader verify:

\begin{lemma}
	The flow $\varphi$ is well-defined on $M$ and continuous. Additionally, $\frac{\partial\varphi_t}{\partial t}$ is well-defined and continuous in restriction to any flow lines, but not continuous globally.
\end{lemma}

As we prove below, $\varphi$ admits a (continuous only) partial cross-section. But it admits no partial cross-section of class $\Class^1$, even only topologically transverse to the flow. 

\begin{lemma}
	The flow $\varphi$ admits a partial cross-section that contains $(0,0)$.
\end{lemma}

\begin{proof}
	We build partial cross-section $S$ as follows. Take a partial cross-section $S_N$ of $N$ that goes from $(0,0)$ to $(1,0)$. Then define $S$ to be the union of $(0,0)$ with the sets $f^\sigma_n(S_N)$ for all $n$ and $\sigma$. Since $f^\sigma_n$ contracts the $y$-coordinates by a factor that goes to zero in $n$, $S$ is indeed a partial cross-section of $\varphi$.
\end{proof}

\begin{lemma}
	There is no $\Class^1$ embedded curve topologically transverse to $\varphi$ whose interior intersects $\{0\}\times\RR$. 
\end{lemma}

\begin{proof}
	Take a $\Class^1$ curve $\delta\colon[-1,1]\to M$ with $\delta(0)=(0,y)$ for some $y$, and so that $\frac{\partial\delta(t)}{\delta t}_{|0}$ is equal to a non-zero vector $(a,b)$. Since $\delta$ is topologically transverse to $\varphi$, it intersects the two connected components of $M\setminus\{0\}\times\RR$.

	So for all large enough $n$, $\delta$ intersects $\im(f^\sigma_n)$. Then the derivative of $(f^\sigma_n)^{-1}(\delta)$ is very close to $(\sigma\frac{3^na}{2},2^nb)$. 

	Assume first that $a\neq 0$ holds. The derivative is very close to being horizontal, so when $n$ is large enough $(f^\sigma_n)^{-1}(\delta)$ intersects $\gamma_0$ with the two possible co-orientations. So $\delta$ it is not topologically transverse to $\varphi$.

	Assume now $a=0$, and so $b\neq0$. If $b$ is positive, then similarly, $(f^+_n)^{-1}(\delta)$ has a slope very large. So it intersects $\gamma_0$ with both signs. If $b$ is negative, the same holds for $(f^-_n)^{-1}(\delta)$. Therefore, in every case, $\delta$ is not topologically transverse to $\varphi$.
\end{proof}

\section{Fried desingularization}\label{sec-Fried-sum}

Take two partial cross-sections $S_1,S_2$. When there are disjoint, $S_1\cup S_2$ is a partial cross-section. When there are not, Fried \cite{Fried1983} described a way to produce a partial cross-section corresponding to $S_1\cup S_2$. The procedure goes as follows. Put $S_1$ and $S_2$ in general position, so that they intersect topologically transversally. Cut them at their common intersection, and glue them back transversally to the flow. After a small perturbation, it yields a new partial cross-section close to $S_1\cup S_2$. This has a nice interpretation using Theorem~\ref{thm-ps-classification-3}.

Given two classes $\alpha_1,\alpha_2$ so that $\PS_\varphi(\alpha_2)$ and $\PS_\varphi(\alpha_2)$ are not empty, the set $\Rec_{\alpha_1+\alpha_2}$ lies inside $\Rec_{\alpha_1}\cap\Rec_{\alpha_2}$ (see \cite[Lemma 5.18]{martyPS1}). A first consequence is that $D_{\varphi,\alpha_1+\alpha_2}\rmod{\germ(\Rec_{\alpha_i})}$ lies inside $D_{\varphi,\alpha_i}$ for $i=1,2$.
If follows a commutative diagram:
$$\begin{tikzcd}
	\PS_\varphi(\alpha_1)\times\PS_\varphi(\alpha_2) \arrow[r]\arrow[d, "\simeqd"] & \PS_\varphi(\alpha_1+\alpha_2) \arrow[d, "\simeqd"] \\
	H^1(M,\germ(\Rec_{\alpha_1}),\ZZ)\times H^1(M,\germ(\Rec_{\alpha_2}),\ZZ) \arrow[r] & H^1(M,\germ(\Rec_{\alpha_1+\alpha_2}),\ZZ),
\end{tikzcd}$$
where the top arrow is the Fried desingularization and the bottom arrow is the sum of the natural restriction maps. 
The bottom arrow sends the pairs $(\beta_1,\beta_2)$ with $\beta_i(D_{\varphi,\alpha_i})\geq0$ onto the pairs $\beta$ with $\beta(D_{\varphi,\alpha_1+\alpha_2})\geq0$, as required by Theorem~\ref{thm-ps-classification-3}.
Note that the top and bottom maps are not necessarily injective nor surjective in general.

For the non-injectivity, take $\alpha_1=0$, assuming the flow non-chain recurrent, and $\alpha_2(D_\varphi)>0$. The top map goes from an infinite set to a singleton, so it is not injective. 
For the non-surjectivity, Example 5.19 in \cite{martyPS1} gives an example where the inclusion $\Rec_{\alpha_1+\alpha_2}\subset\Rec_{\alpha_1}\cap\Rec_{\alpha_2}$ is strict. So any point in $(\Rec_{\alpha_1}\cap\Rec_{\alpha_2})\setminus\Rec_{\alpha_1+\alpha_2}$ lies on a partial cross-section cohomologous to $\alpha_1+\alpha_2$, but it lies on no partial cross-section cohomologous to either $\alpha_1$ or $\alpha_2$ (see Theorem~\ref{thm-ps-rec-disjoint}).

\begin{question}
	When is the map $\PS_\varphi(\alpha_1)\times\PS_\varphi(\alpha_2)\to \PS_\varphi(\alpha_1+\alpha_2)$ injective? and surjective?
\end{question}

Note that for this question, it would be fair to consider the empty set as a null-homologous partial cross-section, otherwise the surjectivity in the case $\alpha_1=0$ will fail quite frequently.

\section{Practical application to simple~$\alpha$-recurrence sets}\label{app-practical-applications}

We give two homological results as practical tools. In the homological classification of partial cross-sections, we utilize the homology relative to the germ of the~$\alpha$-recurrent set $K=\Rec_\alpha$, where $-\alpha$ is quasi-Lyapunov. In many application cases, the compact set~$\Rec_\alpha$ has finitely many connected components. Under that hypothesis, we prove that the homology module is finitely generated.

A neighborhood~$U$ of~$K$ is said to be \emph{exactly $K$-connected} if any path-connected component of~$U$ contains exactly one connected component of~$K$.

\begin{lemma}\label{lem-hom-const}
	Assume that~$K$ has finitely many connected components. Take two neighborhoods $U$ and $V$ of~$K$, assumed~$K$-thin and exactly $K$-connected. Then we have $H_1(U,K,\RR)\lmod{M}=H_1(V,K,\RR)\lmod{M}$.
\end{lemma}

When~$K$ has finitely many connected components, and $\epsilon>0$ is small enough, $V_\epsilon(K)$ is~$K$-thin and exactly $K$-connected. So one can apply the lemma above to $V_\epsilon(K)$.

\begin{proof} 
	Here, all coefficients are in $\RR$.
	It is enough to prove the conclusion when $V\subset U$ holds. Under that assumption, it is clear that $H_1(V,K)\lmod{M}$ lies inside $H_1(U,K)\lmod{M}$. We prove the other inclusion. 
	Consider the following commutative diagram:
	\[
	\begin{tikzcd}
		H_1(V) \arrow[r] \arrow[dr]& H_1(U)\oplus H_1(V,K) \arrow[d, shift right=0.21cm]\arrow[r] & H_1(U,K) \arrow[d]\arrow[r, "\partial"] \arrow[dr, "0"] & H_0(V) \arrow[d, "\simeqd"]\\
		& H_1(M)\oplus H_1(V,K) \arrow[r] & H_1(M,K) & H_0(U)		 
	\end{tikzcd}
	\]
	The upper line is obtained from a relative Mayer-Vietoris sequence. The map $H_0(V)\to H_0(U)$ is an isomorphism since $U$ and $V$ are both exactly $K$-connected.
	Take an element $\delta$ in $H_1(U,K)$. Its boundary $H_0(U)$ is zero. So from the discussion above, its image in $H_0(V)$ is also zero. Hence, there exists an element $(x,y)$ in $H_1(U)\oplus H_1(V,K)$ in the preimage of $\delta$. That is $x\rmod{K}+y\lmod{U}=\delta$ holds. Since $U$ is $K$-thin, there exists $z$ in $H_1(V)$ which satisfies $z\lmod{M}=x\lmod{M}$. Then we have
	$$\delta\lmod{M}=z\hmod{M,K}+y\hmod{M,K}=(z\rmod{K}+y)\lmod{M},$$
	where the class $z\rmod{K}+y$ lies in $H_1(V,K)$.
	It follows that $\delta\lmod{M}$ lies in $H_1(V,K)\lmod{M}$.
\end{proof}



It has the following consequences.

\begin{proposition}\label{prop-germ-char-R}
	Assume that~$K$ has finitely many connected components. Let~$U$ be a~$K$-thin and exactly $K$-connected neighborhood of~$K$. Then the maps 
	$$H_1(M,\germ(K),\RR)\xrightarrow{\srmod} H_1(M,U,\RR)$$
	and 
	$$H^1(M,U,\RR)\xrightarrow{\srmod} H^1(M,\germ(K),\RR)$$
	are isomorphisms.
\end{proposition}

\begin{proposition}\label{prop-germ-char-Z}
	Assume that~$K$ has finitely many connected components. Let~$U$ be a~$K$-thin and exactly $K$-connected neighborhood of~$K$. Then the maps 
	$$\tor(H_1(M,\germ(K),\ZZ))\xrightarrow{\srmod} \tor(H_1(M,U,\ZZ))$$
	and 
	$$H^1(M,U,\ZZ)\xrightarrow{\srmod} H^1(M,\germ(K),\ZZ)$$
	are isomorphisms. If additionally~$U$ is a small enough neighborhood of~$K$, then the map 
	$$H_1(M,\germ(K),\ZZ)\xrightarrow{\srmod}H_1(M,U,\ZZ)$$
	is an isomorphism.
\end{proposition}

It immediately implies:

\begin{corollary}
	Assume that~$K$ has finitely many connected components. Then for $\A=\RR$ or $\ZZ$, the modules $H_1(M,\germ(K),\A)$ and $H^1(M,\germ(K),\A)$ are finitely generated.
\end{corollary} 

\begin{proof}[Proof of Proposition~\ref{prop-germ-char-R}]
	Take two neighborhoods $V\subset U$ of~$K$, assumed exactly $K$-connected and~$K$-thin. We consider the following diagram.
	\[
	\begin{tikzcd}
		H_1(V,K,\RR) \arrow[d,"\vlmod{M}"] \arrow[r,"\vlmod{U}"] & H_1(U,K,\RR) \arrow[d,"\vlmod{M}"] \\
		H_1(M,K,\RR) \arrow[d,two heads,"\prmod{V}"] \arrow[r,"\id"] & H_1(M,K,\RR) \arrow[d,two heads,"\prmod{U}"] \\
		H_1(M,V,\RR) \arrow[r,two heads,"\prmod{U}"] \arrow[d,"\partial"] & H_1(M,U,\RR) \arrow[r,"\partial"] \arrow[d,"\partial"] & H_0(U,V,\RR)=0 \\
		H_0(V,K,\RR)=0 & H_0(U,K,\RR)=0 \\
	\end{tikzcd}
	\]
	The two vertical lines are exact. They end on $H_0(V,K,\RR)$ and $H_0(U,K,\RR)$, which are both equal to zero by exact $K$-connectedness. Thus, the two middle vertical maps are surjective. A similar argument yields that the map~$\prmod{U}$, on the third horizontal line, is surjective.
	
	Let us show that this map $H_1(M,V,\RR)\xrightarrow{\rmod{U}}H_1(M,U,\RR)$ is injective. Take $x$ in the kernel of that map. 
	From the discussion above, there exists $y$ in $H_1(M,K,\RR)$ with $y\rmod{V}=x$. We have $y\rmod{U}=x\rmod{U}=0$. So by exactness of the right vertical sequence, there exists $z$ in $H_1(U,K,\RR)$ with $z\lmod{M}=y$.
	From Lemma~\ref{lem-hom-const}, there exists $z'$ in $H_1(V,K,\RR)$ which satisfies $z'\rmod{M}=y$. It follows $x=z'\hmod{M,V}=0$ from the left vertical sequence.
	Therefore, $f$ is a bijection. 
	
	Recall that when $\epsilon>0$ is small enough, the $\epsilon$-neighborhood $V_\epsilon(K)$ of $K$ is exactly $K$-connected and~$K$-thin. Additionally, by compactness of~$K$, any neighborhood of~$K$ contains some neighborhood $V_\epsilon(K)$. It follows from the definition of limits that $H_1(M,\germ(K),\RR)\xrightarrow{\rmod{U}} H_1(M,U,\RR)$ is an isomorphism when $U$ is $K$-thin and exactly $K$-connected. 
	
	The isomorphism between the cohomology modules follows by duality (see Proposition~\ref{prop-germ-form}).
\end{proof}

\begin{proof}[Proof of Proposition~\ref{prop-germ-char-Z}]
	Let~$U$ be a neighborhood of~$K$, assumed exactly $K$-connected and~$K$-thin. Consider the following diagram:
	\[
	\begin{tikzcd}[column sep=1.6cm]
		\bfrac{H_1(M,\germ(K),\ZZ)}{\tor} \arrow[d,hook] \arrow[r,two heads, "\prmod{U}"] & \bfrac{H_1(M,U,\ZZ)}{\tor} \arrow[d,hook] \\[-0.1cm]
		H_1(M,\germ(K),\RR) \arrow[r, "\simeq"] & H_1(M,U,\RR) 
	\end{tikzcd}
	\]
	By Lemma~\ref{lem-germ-surf}, the map $\prmod{U}$ is surjective.
	It follows Lemma~\ref{lem-coeff-extension} that the two vertical maps are injective. And from Proposition~\ref{prop-germ-char-R}, the bottom map is an isomorphism. It follows that $\prmod{U}$ is an isomorphism. 

	The isomorphism between the cohomology modules follows from Proposition~\ref{prop-germ-form}.
	When $U$ is small enough, the isomorphism between the homology modules, before the quotient by the torsion, follows from Lemma~\ref{lem-tor-bounded}.
\end{proof}

\addcontentsline{toc}{section}{References}
\bibliographystyle{alpha}
\bibliography{ref}

\end{document}